\documentclass{amsart}
\usepackage{amsmath, mathtools, amssymb, cite}
\usepackage{tikz}
\usepackage{pgfplots}
\usetikzlibrary{datavisualization}
\usepackage[matrix,arrow]{xy}
\usepackage{hyperref}
\usepackage{color}
\usepackage[margin=2.4cm]{geometry}
\usepackage{enumerate}
\newtheorem{prop}{Proposition}[section]
\newtheorem{theorem}[prop]{Theorem}
\newtheorem{lemma}[prop]{Lemma}
\newtheorem{cor}[prop]{Corollary}

\newtheorem{conjecture}[prop]{Conjecture}
\theoremstyle{definition}
\newtheorem{dfn}[prop]{Definition}

\newtheorem{rmk}[prop]{Remark}
\newcommand{\git}{\mathrm{/\hskip-3pt/}}

\newcommand{\PP}[0]{\mathbb{P}}
\renewcommand{\geq}{\geqslant}

\begin{document}
	\title[The moduli continuity method for log K-stable pairs]{Applications of the moduli continuity method to log K-stable pairs}
	\author[Gallardo]{Patricio Gallardo}
	\email{pgallard@ucr.edu}
	\address{Department of Mathematics\\University of California Riverside\\Riverside, CA 92521, United States}
	\curraddr{Department of Mathematics\\Washington University\\St. Louis, MO 63130-4899, United States}
	\author[Martinez-Garcia]{Jesus Martinez-Garcia}
	\email{jesus.martinez-garcia@essex.ac.uk}
	\address{Department of Mathematical Sciences\\University of Essex\\Colchester, Essex, CO4 3SQ, United Kingdom}
	\author[Spotti]{Cristiano Spotti}
	\email{c.spotti@qgm.au.dk}
	\address{Department of Mathematics\\Centre for Quantum Geometry of Moduli Spaces, Aarhus Universitet, Aarhus, Denmark}
	\subjclass[2010]{32Q20, 14J45, 14D22, 14J10, 14L24}
	\begin{abstract}
		The `moduli continuity method' permits an \emph{explicit} algebraisation of the Gromov-Hausdorff compactification of K\"ahler-Einstein metrics on Fano manifolds in some fundamental examples. In this paper, we apply such method in the `log setting' to describe explicitly some compact moduli spaces of K-polystable log Fano pairs. We focus on situations when the angle of singularities is perturbed in an interval sufficiently close to one, by considering constructions arising from Geometric Invariant Theory. More precisely, we discuss the cases of pairs given by cubic surfaces with anticanonical sections, and of projective space with non-Fano hypersurfaces, and we show ampleness of the CM line bundle on their good moduli space (in the sense of Alper). Finally, we introduce a conjecture relating K-stability (and degenerations) of log pairs formed by a fixed Fano variety and pluri-anticanonical sections to certain natural GIT quotients.
	\end{abstract}
	\maketitle
	
	%\tableofcontents
	
	\section{Introduction}
	
	Deciding when a Fano variety, and more generally a log Fano pair $(X,(1-\beta)D)$, is log K-stable remains an interesting open problem. In particular, in the logarithmic setting it is natural to ask how log K-stability varies when $\beta$ is perturbed. This question was initially addressed for del Pezzo surfaces and anticanonical divisors in \cite{Cheltsov-JMG-dynamic}, where no moduli considerations were given.

	 In the present article  we want to describe some prototypical situation in the logarithmic setting in which the study of log K-stability compactifications can be reduced to GIT problems, by using the ``moduli continuity method'' strategy already implemented  in the absolute case \cite{Odaka-Spotti-Sun, Spotti-Sun-delPezzo-quadrics, Liu-Xu-Kstability-cubic-threefolds}, and pioneered in \cite{Mabuchi-Mukai}. Beside reducing the checking of K-stability (and hence of the existence of K\"ahler-Einstein metrics) to an explicitly checkable condition, this approach has the advantage of classifying continuous families of compact moduli spaces of K-stable log Fano pairs in some situations.
	
	The first case we are going to analyse consists of pairs given by a del Pezzo surface of degree $3$ and an anticanonical divisor. More precisely, we are interested in studying the moduli spaces of $\mathbb{Q}$-Gorenstein \emph{smoothable} K-stable pairs as above. Note that when such a del Pezzo surface is smooth (or more generally it has canonical singularities), it is realised in $\mathbb{P}^3$ as a cubic hypersurface, and its anticanonical sections are given by hyperplane sections. In \cite{Gallardo-JMG-framework} notions of GIT$_t$-stability for log pairs given by a Fano or Calabi-Yau hypersurface and a hyperplane section were introduced. This notion of stability depends on a parameter $t$ moving in an open interval and moduli spaces of such log pairs were constructed using GIT$_t$-stability. Furthermore, in \cite{Gallardo-JMG-cubic-surfaces} all the GIT compactifications of log pairs formed by a cubic surface and an anticanonical divisor were described.  We will use one of these GIT$_t$-compactifications to realise the log K-stability compactification where the angle of conical singularities of the log pairs is large.
	
	Our first result shows that for  $\beta$ sufficiently close to one, K-stability reduces to the above GIT$_{t(\beta)}$-stabilities.
	
	\begin{theorem}
		\label{theorem:main-cubicsurfaces}
		If $\beta > \beta_0=\frac{\sqrt{3}}{2}$, then we have a natural homeomorphism between the Gromov-Hausdorff compactification of the moduli spaces $\overline{M}^K_{3,\beta}$ of  Kahler-Einstein/K-polystable pairs $(C,(1-\beta)D)$ (where $C$ is a cubic surface and $D$ is an anticanonical section) and the GIT$_t$-quotients $\overline{M}^{GIT}_{t(\beta)}$, for the explicit algebraic function $t(\beta)=\frac{9(1-\beta)}{9-\beta}$, with inverse $\beta(t)=\frac{9(1-t)}{9-t}$.
		
		Let $\pi$ be the GIT quotient morphism with target $\overline{M}^{GIT}_{t(\beta)}$ (as defined in \cite{Gallardo-JMG-framework, Gallardo-JMG-cubic-surfaces}, cf. Section \ref{sec:GIT-hypersurfaces}) and $\pi_1$ and $\pi_2$ be the natural projections from the space of embedded pairs $(C, H)\subset \mathbb P^3$ --- where $H$ is the hyperplane in $\mathbb P^3$ determining $D$ ---  to the Hilbert schemes of cubic surfaces and hyperplanes in $\mathbb P^3$, respectively. For any $a,b>0$, consider the ample line bundle 
		$$\mathcal O(a,b)\coloneqq \pi_*(\pi_1^*(\mathcal O(a))\otimes\pi_2^*(\mathcal O(b))),$$
		on $\overline{M}^{GIT}_{t(\beta)}$. Then the canonically defined log CM line $\Lambda_{CM, \beta(t)}$ on $\overline M^K_{3,\beta}$ is isomorphic to $\mathcal O(a,b)$ for some $a,b>0$ such that $\frac{a}{b}=t(\beta)$.
	\end{theorem}
	The CM line bundle is canonically defined on the base scheme of families of K-polystable varieties. Introduced by Paul and Tian in the absolute case \cite{paul-tian-cm1}, its definition can naturally be extended for pairs (see \S\ref{sec:DFandGIT}). It has been conjectured that the CM line bundle is ample on the good moduli space (in the sense of Alper \cite{Alper-good-moduli}) in the absolute case \cite{Odaka-Spotti-Sun}. This conjecture can be naturally extended to the log CM line bundle in the logaritmic setting. A consequence of Theorem \ref{theorem:main-cubicsurfaces} is that the ampleness condition is verified in the instance of $\overline M^K_{3,\beta}$ when $\frac{\sqrt{3}}{2}<\beta<1$.
	
	Our result identifies $\overline{M}^{GIT}_{t(\beta)}$ with $\overline{M}^K_{3,\beta}$ for $1>\beta>\beta_0\coloneq\frac{\sqrt{3}}{2}$. One would expect to obtain the log K-stability compactification  $\overline{M}^K_{3,\beta}$ for smaller angles $0<\beta<\beta_0$ by modifying appropriately the GIT$_t$-stability compactifications $\overline{M}^{GIT}_{t(\beta)}$ constructed in \cite{Gallardo-JMG-cubic-surfaces} by performing several birational transformations of $\overline{M}^{GIT}_{t(\beta)}$. The expected correspondence between GIT$_t$-stability and log K-stability for $0<\beta<\beta_0$ is made explicit in Theorem \ref{theorem:git-cm-correspondence} by defining a function $t(\beta)$ which identifies the natural polarizing line bundle of $\overline{M}^{GIT}_{t(\beta)}$ and the CM line bundle of $\overline{M}^K_{3,\beta}$. This identification generalises the one in Theorem \ref{theorem:main-cubicsurfaces} to all $0<\beta<1$.
	
	We should stress that, in order to find good a-priori bounds on the singularities of the Gromov-Hausdorff limits, the proof of the above results makes essential use of recent advances on bounds of the so-called normalised volumes of singularities of K-stable varieties \cite{Li-Liu-Xu-normalized-volume-guide}.
	
	Next we consider some higher dimensional examples. A natural Gap Conjecture \cite[Conjecture 5.5, cf. Conjecture 2.1]{Spotti-Sun-delPezzo-quadrics} \ref{conjecture:gap} about normalised volumes says that (in the absolute case) there are no Kawamata log terminal (klt) singularities whose normalised volume is bigger than the one of the $n$-dimensional ordinary double point singularity.  A proof of the Gap Conjecture is available in dimensions $2$ and $3$ \cite{Liu-volume-bound-surfaces, Liu-Xu-Kstability-cubic-threefolds}, and it has been used in \cite{Spotti-Sun-delPezzo-quadrics, Liu-Xu-Kstability-cubic-threefolds} to study compactifications of K-stable varieties. Our next result shows that, assuming such conjecture, the K-stability of pairs $(\mathbb{P}^n, (1-\beta)H)$ with $H$ a hypersurface of degree $d\geq n+1$ reduces to the classical GIT-stability of hypersurfaces for $\beta$ sufficiently close to one. More precisely:

	\begin{theorem}
		\label{theorem:main-Pn}
		If the Gap Conjecture \ref{conjecture:gap} (introduced in \cite[Conjecture 5.5]{Spotti-Sun-delPezzo-quadrics}) holds, then for all $d\geq n+1$ there exists $\beta_0=\beta_0(n,d)$ such that for all $\beta \in (\beta_0, 1)$, the pair $(\mathbb{P}^n,(1-\beta)H_d)$, with $H_d$ any possibly singular degree $d$ hypersurface, is log K-polystable if and only if $[H_d]$ is GIT-polystable for the natural action of $\mathrm{SL}(n+1)$ on $\mathbb{P}(H^0(\mathcal{O}(d)))$. In particular, this holds in dimension $n=1,2,3$, thanks to \cite{Liu-volume-bound-surfaces, Liu-Xu-Kstability-cubic-threefolds}.
		
		A (non-optimal) \emph{explicit} choice of $\beta_0$ is given by
		\begin{equation}
		\label{eq:beta0-Pn}
		\beta_0= 1-\Bigg(\frac{n+1}{d}\big(1-\sqrt[n]{2}(1-\frac{1}{n})\big)\Bigg).
		\end{equation}
		
		Moreover, if $\beta_0<\beta<1$, the Gromov-Hausdorff compactification of K-polystable pairs $(\mathbb P^n, (1-\beta)H_d)$ is homeomorphic to the GIT quotient $\overline{M}^{GIT}_d$ of degree $d$ hypersurfaces in $\mathbb P^n$.
	\end{theorem} 
	
	The first statement of the above result is a special case of the following natural expectation, at least when the automorphism group has no non-trivial characters.
	\begin{conjecture}
		\label{conjecture:main}
		Let $X$ be a K-polystable Fano variety. Then for any sufficiently large and divisible $l\in \mathbb{N}$, there exists $\beta_0=\beta_0(l, X)$ such that for all $ D \in \vert -l K_X \vert$ and $\beta \in (\beta_0, 1)$, the pair
		$ (X, (1-\beta) D)$ is log K-polystable if and only if $[D] \in \mathbb{P}(H^0(-lK_X))$ is GIT-polystable for the natural representation of $\mathrm{Aut}(X)$ on $H^0(-lK_X)$. 
		
		Moreover, the Gromov-Hausdorff compactification $\overline M^K_{X,l,\beta}$ of the above pairs $(X,(1-\beta)D)$  is homeomorphic to the GIT quotient $\mathbb P(H^0(-lK_X))^{ss}\git \mathrm{Aut}(X)$.
	\end{conjecture}
	Note that  the automorphism group  $\mathrm{Aut}(X)$ of a K-polystable Fano is always reductive. If $X$ is smooth or a Gromov-Hausdorff degeneration, this follows from  by Matsushima's obstruction \cite{Matsushima-obstruction} and Chen-Donaldson-Sun \cite{Chen-Donaldson-Sun-Kstability-all}. Otherwise, this has been very recently established in \cite{ABHLX} via purely algebro-geometric  techniques
	
	Once its different steps are established, the application of the moduli continuity method (the proof of theorems \ref{theorem:main-cubicsurfaces} and \ref{theorem:main-Pn}) is just a couple of paragraphs. However, in order to keep the reader focused on the main goal of the paper, we recall here what the different steps of the moduli continuity method are, noting that each of them requires involved technical proofs to accomplish them. Suppose that we have a Gromov-Hausdorff compactification space $\overline{M}^{GH}$ of objects (varieties, log pairs...) admitting a K\"ahler-Einstein type of metric (usually smooth or with conical singularities). Suppose further that we want to show $\overline{M}^{GH}$ is homeomorphic to certain algebraic scheme, usually obtained as a GIT quotient $\overline{M}^{GIT}=\mathcal H^{ss}\git G$, where $\mathcal H$ is some scheme representing objects which may or may not admit a metric of K\"ahler-Einstein type. The moduli continuity method would proceed as follows (see brackets for where each step is accomplished for theorems \ref{theorem:main-cubicsurfaces} and \ref{theorem:main-Pn}, respectively):
	\begin{enumerate}[(1)]
		\item Construct $\overline {M}^{GIT}$ and possibly characterise geometrically (e.g. in terms of their singularities) the (poly /semi)\-stable elements $p$ represented as $[p]\in \mathcal H$ (\cite{Gallardo-JMG-framework, Gallardo-JMG-cubic-surfaces}; \cite{MumfordGIT}).
		\item Show (perhaps using the geometric characterization in (1)) that all elements in $\overline{M}^{KE}$ are represented in $\mathcal H$ (Proposition \ref{prop:GH-limit-cubic pairs}; Proposition \ref{prop:GH-Pn-hypersurfaces}).
		\item Show that for any $[p]\in \mathcal H$ such that the object $p$ is K-(poly/semi)stable, the point $[p]\in \mathcal H$ is GIT (poly/semi)stable.
		\item As a result of steps (1)--(3) one has a natural map $\phi\colon \overline {M}^{GH}\rightarrow \overline{M}^{GIT}$, which one must show to be injective (usually follows from uniqueness of K\"ahler-Einstein metrics) and continuous (usually follows from \cite{Chen-Donaldson-Sun-Kstability-all} and Luna slice theorem).
		\item Moreover, one must show that $\mathrm{Im}(\phi)$ is open and dense (usually this depends on the choice of GIT quotient).
		\item The rest is just an application of topological facts: $\mathrm{Im}(\phi)$ is compact as $\phi$ is continuous and $\overline{M}^{GH}$ is compact. As $\mathrm{Im}(\phi)$ is also dense, it follows that $\phi$ is surjective and in fact a homeomorphism as $\phi$ is a continuous map between a compact space and a Hausdorff space.
	\end{enumerate}
	
%	Furthermore, in view of the second statement of Theorem \ref{theorem:main-Pn}, we can expect the stronger moduli version of Conjecture \ref{conjecture:main}:	
%	\begin{conjecture}
%		\label{conjecture:main-moduli}
%		Let $X$ be a K-polystable Fano variety and assume that $\mathrm{Aut}(X)$ is reductive. Then for any sufficiently large and divisible $l\in \mathbb{N}$, there exists $\beta_0=\beta_0(l, X)$ such that for all $\beta\in(\beta_0, 1)$ the Gromov-Hausdorff compactification $\overline M^K_{X,l,\beta}$ of pairs $(X,(1-\beta)D)$, where  $ D \in \vert -l K_X \vert$ is homeomorphic to the GIT quotient $\mathbb P(H^0(-lK_X))^{ss}\git \mathrm{Aut}(X)$ for the natural representation of $\mathrm{Aut}(X)$ on $H^0(-lK_X)$. 
%	\end{conjecture} 

	The structure of the paper is as follows. In section \ref{sec:DFandGIT} we discuss some relations between log-K-stability and GITs, recalling the relevant definitions. In section \ref{sec:GIT-hypersurfaces}, we apply the machinery to some specific cases, which include the ones relevant for our main theorems listed above. Finally, in section \ref{sec:proofs} we finish the proofs of the main theorems, and discuss further directions. 
	
	\begin{rmk}
	CS and JMG proved Theorem \ref{theorem:main-Pn} during the visit of CS to JMG in August 2018, as part of this project. When communicating the proof to PG, the latter informed them that such a statement in the special case of $n=2$, had been obtained independently by another group pursuing a different question \cite{Ascher-DeVleming-Liu}. Ten months after the appearance of our work in preprint form in the ArXiv, \cite{Ascher-DeVleming-Liu} was shared with the mathematical community in the ArXiv, including a proof of  Theorem \ref{theorem:main-Pn}  for $n\geq 4$ without assuming the Gap Conjecture, as well as studying in details wall-crossing phenomena, predicted in section \ref{wall} in relation to the pairs given by $\mathbb{P}^2$ and a quartic. The paper  \cite{Ascher-DeVleming-Liu} also includes the description of a natural algebraic structure on the moduli space of general smoothable K-polystable Fano pairs, using analytic techniques, in analogy to the absolute case established in  \cite{Spotti-Sun-Yao-singular-KE-Kstability}, \cite{Odaka-compactification} and \cite{Li-Wang-Xu-compact-moduli}. Regarding this last point, it is worth mentioning that recently there has been several important advances in establishing properties of the moduli of (non-necessarily smoothable) Fano (pairs) via purely algebro-geometric techniques, such as its separatedness \cite{Blum-Xu}, its existence \cite{ABHLX} (but not yet properness), and the general openness of K-semistability (\cite{Xu20} and \cite{BLX}).
	\end{rmk}
	
	\subsection*{Acknowledgments}
	PG  is grateful for the working environment of the Department of Mathematics in Washington University at St. Louis. PG's travel related to this project was partially covered by the FRG Grant DMS-1361147 (P.I Matt Kerr). JMG is supported by the Simons Foundation under the Simons Collaboration on Special Holonomy in Geometry, Analysis and Physics (grant \#488631, Johannes Nordstr\"om). CS is supported by AUFF Starting Grant 24285, DNRF Grant DNRF95 QGM `Centre for Quantum Geometry of Moduli Spaces', and by Villum Fonden 0019098. This project was started at the Hausdorff Research Institute for Mathematics (HIM) during a visit by the authors as part of the Research in Groups project \emph{Moduli spaces of log del Pezzo pairs and K-stability}. We thank HIM for their generous support.

	We would like to thank F. Gounelas for a clarification regarding the Fano index in deformation families. We would like to thank M. de Borb\'on, J. Nordstr\"om, Y. Odaka and S. Sun for useful comments.
	 
	\section{Donaldson-Futaki invariant and Geometric Invariant Theory}\label{sec:DFandGIT}
	Let $\pi\colon\mathcal X\rightarrow \mathcal B$ be a flat proper morphism of relative dimension $n$. Let $\mathcal D\subseteq \mathcal X$ be an effective Weil $\mathbb Q$-divisor of $\mathcal X$ such that $\mathcal D|_b$ is equidimensional of dimension $n-1$ for all $b\in \mathcal B$. Let $\mathcal L$ be a $\pi$-very ample $\mathbb Q$-line bundle of $\mathcal X$. Assume that the restriction $\pi|_{\mathcal D}\colon \mathcal D \rightarrow \mathcal B$ is also a flat proper morphism of relative dimension $n-1$ and $\mathcal L_{\mathcal D}$ is $\pi|_{\mathcal D}$-ample.  The fibers of $\pi$ are projective varieties.
	
	For sufficiently large $k>0$, the Knudsen-Mumford theorem \cite{knudsen-mumford} says that there exist functorially defined line bundles $\lambda_j\coloneqq\lambda_j(\mathcal X, \mathcal B, \mathcal L)$ and $\widetilde \lambda_j\coloneqq\lambda_j(\mathcal D, \mathcal B, \mathcal L|_{\mathcal D})$ on $\mathcal{B}$ such that 
	\begin{align*}
	\det\left(\pi !_*\left(\mathcal L^k\right)\right)&=\lambda_{n+1}^{\otimes{\binom{k}{n+1}}}\otimes\lambda_n^{\otimes{\binom{k}{n}}}\otimes\cdots,\\
	\det\left(\pi !_*\left(\left(\mathcal L|_{\mathcal D}\right)^k\right)\right)&=\widetilde\lambda_{n}^{\otimes{\binom{k}{n}}}\otimes\widetilde\lambda_{n-1}^{\otimes{\binom{k}{n-1}}}\otimes\cdots.
	\end{align*}
	
	\begin{lemma}[{\cite{Zhang-Knudsen-Mumford-pairing,phong-ross-sturm},\cite[I.3.1]{elkik}}]
		\label{lemma:deligne}
		For each $r>0$, the line bundles in the Knudsen-Mumford expansion of $\mathcal L^r$ satisfy the following properties in relation to the Deligne's pairing with $n+1$ entries:
		\begin{enumerate}[(i)]
			\item $\lambda_{n+1}(\mathcal L^r)=\langle\mathcal L^r,\cdots,\mathcal L^r \rangle=\left\langle\mathcal L, \ldots, \mathcal L\right\rangle^{\otimes r(n+1)}=\lambda_{n+1}^{\otimes r^{n+1}}.$
			\item If $\mathcal X$ and $\mathcal B$ are smooth, then $(\lambda_n(\mathcal L^r))^{\otimes 2}=\langle \mathcal L^{nr}\otimes K_{\mathcal X/\mathcal B}^{-1}, \mathcal L^r,\ldots, \mathcal L^r\rangle$.
			\item If $\gamma:\mathcal B'\rightarrow \mathcal B$ is a proper morphism and $\mathcal X \times_{\mathcal B}\mathcal B' \rightarrow \mathcal X$ is the pullback of $\gamma$ via the fibred product, then
			$$\big\langle \gamma^*(\mathcal L),\ldots, \gamma^*(\mathcal L)\big\rangle = \gamma^*\big(\langle \mathcal L,\ldots,\mathcal L\rangle\big).$$
			Moreover, Deligne's pairing is multilinear with respect to the tensor product of line bundles.
		\end{enumerate}
	\end{lemma}
	
	Since $\pi$ is flat, the Hilbert polynomial is constant along fibres $b\in \mathcal B$. Let $p(k)$ and $\widetilde p(k)$ be the Hilbert polynomials of $\mathcal L_b$ and $\mathcal L_b|_{\mathcal D}$ on fibres $\mathcal X_b$ and $\mathcal D_b$, respectively. For $k$ sufficiently large, we have
	$$p(k)=a_0k^n+a_1k^{n-1}+\cdots,\qquad \widetilde p(k)=\widetilde a_0 k^{n-1}+\widetilde a_{1}k^{n-2}+\cdots.$$
	If the general fibre $\mathcal X_b$ has mild singularities (e.g. if $\mathcal X_b$ is $\mathbb Q$-factorial), then we can write the coefficients of the Hilbert polynomials in terms of first chern classes:
	\begin{align}
	a_0&=\frac{c_1(\mathcal L_b)^n}{n!}, \label{eq:Hilbert-coeff-a0}\\
	a_1&=\frac{c_1(\mathcal X_b)\cdot c_1(\mathcal L_b)^{n-1}}{2(n-1)!},\label{eq:Hilbert-coeff-a1}\\
	\widetilde a_0&=\frac{c_1(\mathcal L|_{\mathcal D_b})^{n-1}}{(n-1)!}=\frac{c_1(\mathcal L)^{n-1}\cdot \mathcal D_b}{(n-1)!}. \label{eq:Hilbert-coeff-a0-tilde}
	\end{align}
	\begin{dfn}
		Given the tuple $(\mathcal X, \mathcal D, \mathcal B, \mathcal L^r)$ as above, we define its \emph{log CM $\mathbb Q$-line bundle with angle $\beta\in \mathbb Q_{>0}$} on $\mathcal B$ as 
		$$\Lambda_{CM, \beta}(\mathcal X, \mathcal D, \mathcal L)=\lambda_{n+1}^{\otimes\left(n(n+1)+\frac{2a_1-(1-\beta)\widetilde a_0}{a_0}\right)}\otimes\lambda_n^{\otimes(-2(n+1))}\otimes\widetilde \lambda_n^{\otimes (1-\beta)(n+1)}.$$
		We will write either $\Lambda_{CM, \beta}(\mathcal L)$ or $\Lambda_{CM, \beta}$ for $\Lambda_{CM, \beta}(\mathcal X, \mathcal D, \mathcal L)$ whenever there is no confusion on $(\mathcal X, \mathcal D)$ or $\mathcal L$, respectively. 
	\end{dfn}
	When $\beta=1$, the latter definition recovers the definition of CM line bundle for varieties in \cite{paul-tian-cm1, paul-tian-cm2}, c.f. \cite{Li-Wang-Xu-compact-moduli}.
	\begin{lemma}Let $\mathcal X$ and $\mathcal B$ be smooth and assume the general fibre of $\mathcal X_b$ is $\mathbb Q$-factorial. Then $\Lambda_{CM,\beta}(\mathcal L^r)=\big(\Lambda_{CM, \beta}(\mathcal L)\big)^{\otimes r^n}$ for all $r>0$.
	\end{lemma}
	\begin{proof}
		From Lemma \ref{lemma:deligne} (i) we have that $\lambda_{n+1}(\mathcal L^r)=\lambda_{n+1}^{\otimes r^{n+1}}$ and $\widetilde\lambda_{n}(\mathcal L^r)=\lambda_{n}^{\otimes r^{n}}$.
		
		Moreover, Lemma \ref{lemma:deligne} (ii) gives
		\begin{align*}
		(\lambda_n(\mathcal L^r))^{\otimes 2}&=\langle \mathcal L^{nr}\otimes K_{\mathcal X}^{-1}, \mathcal L^r,\ldots, \mathcal L^r\rangle\\
		&=\left[\langle \mathcal L^{nr-n}\otimes \mathcal L^n\otimes K_{\mathcal X}^{-1}, \mathcal L, \ldots, \mathcal L\rangle\right]^{\otimes r^n}\\
		&=\left[\langle \mathcal L^{nr-n}, \mathcal L, \ldots, \mathcal L\rangle\otimes \langle \mathcal L^n\otimes K_{\mathcal X}^{-1}, \mathcal L, \ldots, \mathcal L\rangle\right]^{\otimes r^n}\\
		&=\left[\langle \mathcal L, \ldots, \mathcal L\rangle^{\otimes n(r-1)}\otimes\langle \mathcal L^n\otimes K_{\mathcal X}^{-1}, \mathcal L, \ldots, \mathcal L\rangle\right]^{\otimes r^n}\\
		&=\lambda_{n+1}^{\otimes r^n n (r-1)}\otimes \lambda_n^{\otimes 2r^n}.
		\end{align*}
		Let $a_i(r)$ and $\widetilde a_i(r)$ be the coefficients of the Hilbert polynomials of $\mathcal X$ and $\mathcal D$, respectively, for a power $\mathcal L^r$ of $\mathcal L$. By restricting to a general fibre $b$ of $\mathcal X$ we have
		\begin{align*}
		a_0(r)&=\frac{c_1(\mathcal L^r|_b)^n}{n!}=r^n\frac{c_1(\mathcal L|_b)^n}{n!}=r^n a_0,\\
		\widetilde a_0(r)&=\frac{c_1(\mathcal L^r_{\mathcal D}|_b)^{n-1}}{(n-1)!}=r^{n-1}\frac{c_1(\mathcal L_{\mathcal D}|_b)^{n-1}}{(n-1)!}=r^{n-1} \widetilde a_0,\\ 
		a_1(r)&=\frac{c_1(\mathcal X|_b)\cdot c_1(\mathcal L^r|_b)^{n-1}}{2(n-1)!}=r^{n-1}a_1.
		\end{align*}
		Hence
		\begin{align*}
		\lambda_{n+1}(\mathcal L^r)^{\otimes\left(n(n+1)+\frac{2a_1(r)-(1-\beta)\widetilde a_0(r)}{a_0(r)}\right)}&=(\lambda_{n+1}^{\otimes r^{n+1}})^{\otimes\left(n(n+1)+r^{-1}\frac{2a_1-(1-\beta)\widetilde a_0}{a_0}\right)}\\
		&=\left(\lambda_{n+1}^{\frac{2a_1-(1-\beta)\widetilde a_0}{a_0}}\otimes \lambda_{n+1}^{\otimes rn(n+1)}\right)^{\otimes r^n}.
		\end{align*}
		The latter identity, together with $\big(\lambda_n(\mathcal L^r)\big)^{\otimes 2}=\lambda_{n+1}^{\otimes r^n n (r-1)}\otimes \lambda_n^{\otimes 2r^n}$ completes the proof:
		\begin{align*}
		\Lambda_{CM,\beta}(\mathcal L^r)&=\left(\lambda_{n+1}(\mathcal L^r)\right)^{\otimes n(n+1)+\frac{2a_1(r)-(1-\beta)\widetilde a_0(r)}{a_0(r)}}\\
		&\qquad\	\otimes\left(\lambda_n(\mathcal L^r)\right)^{\otimes(-2(n+1))}\otimes \left(\widetilde \lambda_n(\mathcal L^r)\right)^{\otimes(1-\beta)(n+1)}\\
		&=\left(\lambda_{n+1}^{\otimes\frac{2a_1-(1-\beta)\widetilde a_0}{a_0}}\otimes \lambda_{n+1}^{\otimes rn(n+1)}\right)^{\otimes r^n}\\
		&\quad\ \otimes\left(\lambda_{n+1}^{\otimes(-(n+1)n(r-1))}\right)^{\otimes r^n}\otimes \left(\lambda_n^{\otimes(-2(n+1))}\right)^{\otimes r^n}\otimes\left(\widetilde \lambda_n^{\otimes(1-\beta)(n+1)}\right)^{\otimes r^n}\\
		&=\big(\Lambda_{CM, \beta}\left(\mathcal L\right)\big)^{\otimes r^n}.
		\end{align*}
	\end{proof}
	
	Hence, the log CM $\mathbb Q$-line bundle is unique up to scaling. In particular, if it is ample, it induces the same polarisation of $\mathcal B$ for all $r$ sufficiently large. Therefore we will simply refer to the \emph{log CM line bundle} $\Lambda_{CM, \beta}$.
	
	\begin{dfn}\label{def:test-configuration}
		Let $(X,D)$ be a log pair where $X$ is a projective variety and $D$ is a divisor of $X$. Let $L$ be an ample $\mathbb Q$-line bundle of $X$.
		
		A \emph{test configuration} of $(X,D, L)$ is a tuple  $(\mathcal X, \mathcal D, \mathcal L)$ where $\mathcal D$ is a Weil $\mathbb Q$-divisor of the projective variety $\mathcal X$ and a flat proper morphism $\pi\colon\mathcal X \rightarrow \mathbb C$ such that
		\begin{enumerate}[(i)]
			\item the morphism $\pi$ induces a flat proper morphism $\pi_{\mathcal D}\colon\mathcal D\rightarrow \mathbb C$,
			\item the general fibres of $\pi$ and $\pi_{\mathcal D}$ are isomorphic to $X$ and $D$, respectively,
			\item there is a $\pi$-equivariant $\mathbb C^*$-action on $(\mathcal X, \mathcal L)$ which preserves $\mathcal D$. In particular, all the fibres of $\pi$ and $\pi_{\mathcal D}$ at $s\neq 0\in \mathbb C$ are isomorphic to $X$ and $D$, respectively.
		\end{enumerate}
	\end{dfn}
	
	Let $(\mathcal X, \mathcal D, \mathcal L)$ be a test configuration of $(X,D,L)$. Let $\mathcal X_0$ and $\mathcal D_0$ be the central fibres of $\mathcal X$ and $\mathcal D$. For $k$ sufficiently large, we can write 
	\begin{align*}w(k)=b_0k^{n+1}+b_1k^n+\cdots, \qquad\widetilde w(k)=\widetilde b_0k^n+\cdots,
	\end{align*}
	for the sum of the weights of the action of $\mathbb C^*$ on $\Lambda^{n+1}(H^0(\mathcal X_0, \mathcal L|_{\mathcal X_0}^{\otimes k}))$ and\linebreak $\Lambda^{n}(H^0(\mathcal D_0, \mathcal L|_{\mathcal D_0}^{\otimes k}))$, respectively. Since the Hilbert polynomial is invariant along fibres of flat deformations, \eqref{eq:Hilbert-coeff-a0} \eqref{eq:Hilbert-coeff-a1}, \eqref{eq:Hilbert-coeff-a0-tilde} define its coefficients of the initial terms of the Hilbert polynomials of $X$ and $D$. Let $\beta\in (0,1]\cap \mathbb Q$. The \emph{$\beta$-Donaldson-Futaki invariant} of $(\mathcal X, \mathcal D, \mathcal L)$ is
	$$\mathrm{DF}_\beta(\mathcal X, \mathcal D, \mathcal L)=\frac{2(a_1b_0-a_0b_1)}{a_0}+(1-\beta)\frac{\widetilde b_0a_0-\widetilde a_0b_0}{a_0}.$$
	
	\begin{dfn}
		\label{dfn:Kstability}
		An $L$-polarised pair $(X,(1-\beta)D)$ is \emph{K-semistable} if and only if $\mathrm{DF}_\beta(\mathcal X, \mathcal D, \mathcal L)\geqslant 0$ for all test configurations $(\mathcal X, \mathcal D, \mathcal L)$ of $(X,D,L)$.
		
		An $L$-polarised pair $(X,(1-\beta)D)$ is \emph{K-stable} (respectively \emph{K-polystable}) if and only if $\mathrm{DF}_\beta(\mathcal X, \mathcal D, \mathcal L)>0$ for all test configurations $(\mathcal X, \mathcal D, \mathcal L)$ which are not isomorphic to the trivial test configuration $(X\times \mathbb C, D\times \mathbb C, \pi_1^*(L))$, (respectively equivariantly isomorphic to the trivial test configuration), where $\pi_1$ is projection on the first factor.
	\end{dfn}

	\begin{theorem}
		\label{theorem:weight-DF}
		Let $\mathcal B=\mathbb C$ and suppose that $(\mathcal X, \mathcal D, \mathcal B, \mathcal L)$ is a test configuration of an $L$-polarised pair $(X, D)$. Then
		$$w(\Lambda_{CM, \beta}(\mathcal X, \mathcal D, \mathcal L^r))=(n+1)!\mathrm{DF}_\beta(\mathcal X, \mathcal D, \mathcal L),$$
		where $w(\Lambda_{CM, \beta}(\mathcal X, \mathcal D, \mathcal L^r))$ is the total weight of $\Lambda_{CM, \beta}(\mathcal X, \mathcal D, \mathcal L^r)$ under the $\mathbb C^*$-action of the test configuration.
	\end{theorem}
	\begin{proof}
		Denote by $w(\mathcal G)$ the total weight of the $\mathbb C^*$-action on any $\mathbb C^*$-linearised $\mathbb Q$-line bundle $\mathcal G$. Observe that $\mathcal L$ is $\mathbb C^*$-linearised since $(\mathcal X, \mathcal D, \mathcal L)$ is a test configuration. Then:
		$$w(\det\left(\pi !_*\left(\mathcal L^k\right)\right))=b_0k^{n+1}+ b_1k^n+\ldots.$$
		On the other hand
		\begin{align*}
		w\left(\det\left(\pi !_*\left(\mathcal L^k\right)\right)\right)&=w\left(\lambda_{n+1}^{\otimes{\binom{k}{n+1}}}\otimes\lambda_n^{\otimes{\binom{k} {n}}}\otimes\cdots\right)\\
		&={\binom{k}{n+1}}w(\lambda_{n+1})+{\binom{k}{n}}w(\lambda_n)+\ldots,
		\end{align*}
		and since
		\begin{align*}
		{\binom{k}{n+1}} &= \frac{k^{n+1}}{(n+1)!} - \frac{n(n+1)}{2}\frac{1}{(n+1)!}k^n+\cdots\\
		{\binom{k}{n}} &= \frac{k^{n}}{n!} - \frac{n(n-1)}{2}\frac{1}{n!}k^{n-1}+\cdots,
		\end{align*}
		we have that $b_0=\frac{w(\lambda_{n+1})}{(n+1)!}$. Similarly $\widetilde b_0=\frac{w(\widetilde \lambda_n)}{n!}$. On the other hand, by looking at the coefficients of the $k^n$ terms we get that
		$$b_1=\frac{-n(n+1)}{2}\frac{1}{(n+1)!}w(\lambda_{n+1})+\frac{1}{n!}w(\lambda_n)=\frac{-n(n+1)}{2}b_0+\frac{n+1}{(n+1)!}w(\lambda_n).$$
		Rearranging terms we obtain the weights of each relevant line bundle in the Mumford-Knudsen expansions:
		$$w(\lambda_n)=(n+1)!\left(\frac{b_1}{n+1}+\frac{n}{2}b_0\right),\qquad w(\lambda_{n+1})=(n+1)!b_0,\qquad w(\widetilde \lambda_{n})=n!b_0.$$
		We can now compute the weight of the log CM line bundle, substituting the values for the weights:
		\begin{align*}
		w\left(\Lambda_{CM, \beta}\right)&=\left(n(n+1)+\frac{2a_1-(1-\beta)\widetilde a_0}{a_0}\right)\,w\left(\lambda_{n+1}\right)\\
		&\qquad\qquad-2(n+1)w\left(\lambda_n\right)+\left(1-\beta\right)\left(n+1\right)w(\widetilde \lambda_n)\\
		&=\left(n(n+1)+\frac{2a_1-(1-\beta)\widetilde a_0}{a_0}\right)\,(n+1)!b_0\\
		&\qquad\qquad -2(n+1)(n+1)!\left(\frac{b_1}{n+1}+\frac{n}{2}b_0\right)\\
		&\qquad\qquad +\left(1-\beta\right)\left(n+1\right)n!\widetilde b_0\\
		&=(n+1)!\left(\left(\frac{2a_1b_0-2b_1a_0}{a_0}\right)+(1-\beta)\left(\frac{\widetilde b_0a_0-\widetilde a_0b_0}{a_0}\right)\right)\\
		&=(n+1)!\mathrm{DF}_{\beta}(\mathcal X, \mathcal D, \mathcal L).
		\end{align*}
	\end{proof}

	\begin{theorem}
		\label{theorem:CM-chern-class} Let $(X,D,L)$ be the restriction of a family $(\mathcal X, \mathcal D, \mathcal L)$ where Grothendieck-Riemann-Roch applies  (e.g. if the fibres have mild singularities, for instance if they are locally complete intersections) to a general $b\in \mathcal B$ and let $\mu(L)=\frac{c_1(X)\cdot c_1(L)}{c_1(L)^n}$ and $\mu(L,D)=\frac{D\cdot c_1(L)}{c_1(L)^n}$. Assume that $X$ is $\mathbb Q$-factorial. Then
		\begin{align*}
		\deg(\Lambda_{CM,\beta})=\pi_*&\left(n\mu\left(L\right)  c_1\left(\mathcal L\right)^{n+1}+\left(n+1\right)c_1\left(\mathcal L\right)^nc_1\left(K_{\mathcal X / \mathcal B}\right)+ \right. \\
		&\left.\left(1-\beta\right)\left(\left(n+1\right)c_1\left(\mathcal L\right)^n\cdot \mathcal D - n\mu(L,D)c_1\left(\mathcal L\right)^{n+1}\right)\right).
		\end{align*}
		
		Moreover, if $\mathcal L=-K_{\mathcal X/\mathcal B}$ and $\mathcal D|_{\mathcal X_b}\in |-K_{\mathcal X_b}|$ for all $b\in \mathcal B$, then
		\begin{align*}
		\deg(&\Lambda_{CM,\beta})=\pi_*\left(c_1\left(-K_{\mathcal X/\mathcal B}\right)^n\cdot \left(-c_1\left(-K_{\mathcal X/\mathcal B}\right)+\left(1-\beta\right)\left(\left(n+1\right)\mathcal D-nc_1\left(-K_{\mathcal X/\mathcal B}\right)\right)\right)\right).
		\end{align*}
	\end{theorem}
	\begin{proof}
		From applying the Grothendieck-Riemann-Roch Theorem to $\mathcal L^k$ and $(\mathcal L|_{\mathcal D})^k$ we obtain (c.f. \cite[\S2]{fine2006note}):
		\begin{align*}
		\pi_*(c_1(\mathcal L)^{n+1})&=\deg(\lambda_{n+1})\\
		n\deg(\lambda_{n+1})-2\deg(\lambda_n)&=\pi_*\left(c_1\left(\mathcal L\right)^n \cdot c_1\left(K_{\mathcal X/\mathcal B}\right)\right)\\
		\pi_*\left(c_1\left(\mathcal L|_{\mathcal D}\right)^n\right)&=\deg\left(\widetilde \lambda_n\right).
		\end{align*}
		Hence, we have:
		\begin{align*}
		c_1\left(\Lambda_{CM,\beta}\right)=&\left(n\left(n+1\right)+\frac{2a_1-(1-\beta)\widetilde a_0}{a_0}\right)\pi_*\left(c_1\left(\mathcal L\right)^{n+1}\right)\\
		&\qquad-\left(n+1\right)\pi_*\left(nc_1\left(\mathcal L\right)^{n+1}-c_1\left(\mathcal L\right)^nc\cdot _1\left(K_{\mathcal X/\mathcal B}\right)\right)\\
		&\qquad+(1-\beta)(n+1)\left(\left(\pi|_{\mathcal D}\right)_*\left(c_1\left(\mathcal L|_{\mathcal D}\right)^n\right)\right)\\
		&=\pi_*\left(\left(2\frac{a_1}{a_0}-(1-\beta)\frac{\widetilde a_0}{a_0}\right)c_1\left(\mathcal L\right)^{n+1}+\right. \\
		&\qquad\left(n+1\right)c_1\left(\mathcal L\right)^n\cdot c_1\left(K_{\mathcal X/\mathcal B}\right) + \left(1-\beta\right)(n+1)c_1\left(\mathcal L\right)^n\cdot\mathcal D\bigg)\\
		&=\pi_*\left(2\frac{a_1}{a_0} c_1\left(\mathcal L\right)^{n+1}+\left(n+1\right)c_1\left(\mathcal L\right)^n\cdot c_1\left(K_{\mathcal X/\mathcal B}\right) \right.\\
		&\qquad+\left.\left(1-\beta\right)\left(\left(n+1\right)c_1\left(\mathcal L\right)^n\cdot\mathcal D-\frac{\widetilde a_0}{a_0}c_1\left(\mathcal L\right)^{n+1}\right)\right).
		\end{align*}
		
		We obtain the identities $\frac{2a_1}{a_0}=n\mu(L)$ and $\frac{\widetilde a_0}{a_0}=n\mu(L,D)$ from \eqref{eq:Hilbert-coeff-a0} \eqref{eq:Hilbert-coeff-a1} and \eqref{eq:Hilbert-coeff-a0-tilde}. Substituting in the above identity gives us the first formula. For the second formula note that if $\mathcal L=-K_{\mathcal X/\mathcal B}$, and $\mathcal D|_{\mathcal X_b}\in |-K_{\mathcal X_b}|$, then $\mu(L)=\mu(L,D)=1$ and the second formula follows.

	\end{proof}
	
	\section{Variations of GIT for hypersurfaces}
	\label{sec:GIT-hypersurfaces}
	Let $H_{n,d}=\mathbb P\left(H^0\left(\mathbb P^{n+1}, \mathcal O_{\mathcal P^{n+1}}\left(d\right)\right)\right)$. The product $H_{n,d}\times H_{n,1}$ parametrises, up to multiplication by constants, pairs $(p,l)$ of homogeneous polynomials in $n+2$ variables  where $p$ has degree $d\geqslant 2$ and $l$ has degree $1$. Therefore, it parametrises all pairs $(X,H)$ formed by a (possibly non-reduced and/or reducible) hypersurface of degree $d$ and a hyperplane $H$ in $\mathbb P^{n+1}$. The group $G=\mathrm{SL}(n+1,\mathbb C)$ acts naturally on $H_{n,d}\times H_{n,1}$. Let $\pi_1\colon H_{n,d}\times H_{n,1} \rightarrow H_{n,d}$ and $\pi_2\colon H_{n,d}\times H_{n,1} \rightarrow H_{n,1}$ be the natural projections.
	By \cite{Gallardo-JMG-framework}, $\mathrm{Pic}^G(H_{n,d}\times H_{n,1})\cong \mathbb Z^2$. Moreover we can describe all line bundles $\mathcal L\in \mathrm{Pic}^G(H_{n,d}\times H_{n,1})$ as
	$$\mathcal L\cong \mathcal O(a,b)\coloneqq \left(\pi_1^*\left(\mathcal O_{H_{n,d}}\left(a\right)\right)\otimes \mathcal \pi_2^*\left(O_{H_{n,1}}\left(b\right)\right)\right),$$
	where $a,b\in \mathbb Z$ and $ \mathcal O(a,b)$ is ample if and only if $a>0$ and $b>0$. Therefore we have a natural GIT quotient $( H_{n,d}\times H_{n,1})^{ss}\git_{\mathcal O(a,b)}G$ for each choice of $(a,b)\in \mathbb Z$. Since GIT stability is independent of scaling of $\mathcal O(a,b)$ by a positive constant, we may parametrise the GIT stability conditions by $t=\frac{b}{a}>0$. It follows from \cite[Theorem 1.1]{Gallardo-JMG-framework} that for $t>t_{n,d}\coloneqq \frac{d}{n+1}$ there are no GIT semistable pairs $(X,H)$ and that the interval $(0,t_{n,d})$ is subdivided into a wall-chamber decomposition where walls are given by $t_{1}=0$, $t_2, \ldots, t_k=t_{n,d}$ such that $(H_{n,d}\times H_{n,1})^{ss}\git_{\mathcal O(a,b)}G$ is constant for $t\in (t_i, t_{i+1})$.
	
	The above setting, introduced in \cite{Gallardo-JMG-framework} and expanded in \cite{Gallardo-JMG-Zheng-quartic-curves}, is an application of the more general theory of variations of GIT quotients introduced by Dolgachev and Hu \cite{Dolgachev-Hu-vGIT} and Thaddeus \cite{Thaddeus-vGIT}. The approach introduced in \cite{Gallardo-JMG-framework} is two-fold. On the one hand, an algorithmical approach was introduced to determine the GIT stability of each pair $(X,H)$ via computational geometry. These algorithms were implemented in \cite{Gallardo-JMG-code} to determine the wall-chamber structure, and the GIT unstable, strictly semistable and polystable locus for each value of $t$. On the other hand, these GIT quotients induce a compactification for log pairs in the Fano case. To describe this setting, let $X=\{p(x_0,\ldots,x_{n+1})=0\}$ and $H=\{l=0\}$ for $[p]\in H_{n,d}$, $[l]\in H_{n,1}$. If $\mathrm{Supp}(H)\not\subset\mathrm{Supp}(D)$, or equivalently if $l$ does not divide $p$, then $D:=X\cap H\in \left|\mathbb P\left(H^0\left(X,\mathcal O_X\left(1\right)\right)\right)\right|$ defines a $\mathbb Q$-Cartier divisor on $X$. Moreover, if $d\leqslant n+1$, $X$ is Fano and if $d=n+1$, then $D\in |-K_X|$. 
	
	\begin{theorem}[{\cite[Theorem 1.3]{Gallardo-JMG-framework}}]
		\label{theorem:CY-Fano}
		Every point in the GIT quotient
		$$(H_{n,d}\times H_{n,1})^{ss}\git_{\mathcal O(a,b)}G$$
		parametrises a closed orbit associated to a pair  $(X,D)$ with $D=X \cap H$ in the cases where $X$ is a Calabi-Yau or a Fano hypersurface of degree $d>1$. Furthermore, if $X$ is Fano $t\leqslant t_{n,d}$ and $(X,H)$ is $t$-semistable, then $X$ does not contain a hyperplane in its support, unless $t=t_{n,d}$, in which case $(X,H)$ is strictly $t_{n,d}$-semistable.
	\end{theorem}
	
	In the rest of this section, unless othewise stated, we will assume that $1<d\leqslant n+1$. 
	
	\begin{dfn}
		Let $G=\mathrm{SL}(n+2,\mathbb C)$, $X\subset \mathbb P^{n+1}$ be a Fano hypersurface and $H\subset\mathbb P^{n+1}$ be a hyperplane such that $H\not\subset\mathrm{Supp}(X)$ and $D=X\cap H\sim \mathcal O_{\mathbb P^{n+1}}(1)$. We say that $(X,D)$ is \emph{GIT$_{t}$-semistable  (respectively GIT$_{t}$-stable, GIT$_{t}$-polystable)} if and only if $(X,H)$ is GIT semistable (respectively GIT stable, GIT polystable) with respect to the unique $G$-linearised polarisation $\mathcal O(a,b)$, where $t=\frac{b}{a}$. 
	\end{dfn}

	In order to use the setting from Section \ref{sec:DFandGIT}, we need to construct an appropriate flat family of pairs. Next, we show that such a family exists over an open set $\mathcal U \subset  H_{n,d}\times H_{n,1}$ and that the complement of $\mathcal U$ is the union of two loci $Z_1$ and $Z_2$ with codimension at least two. This will allow us to extend line bundles in $\mathcal U$ to line bundles in $H_{n,d}\times H_{n,1}$.
	
	\begin{lemma}
		\label{lemma:Z1-dim}
		For $n+1\geqslant d \geqslant 2$, let $Z_1=\{(p,l)\subset H_{n,d}\times H_{n,1}  \ | \ X=\{p=0\},\ H=\{l=0\},\ \mathrm{Supp}(H)\subseteq \mathrm{Supp}(X)\}$, then
		
		$$
		\mathrm{codim}(Z_1)=
		\binom{n+d}{d}\geqslant 2.
		$$
	\end{lemma}
	\begin{proof}
		We have $\mathrm{Supp}(H) \subseteq \mathrm{Supp}(X)$ if and only if the polynomials defining the equations of $H$ and $X$ can be written as the homogenous polynomials $l(x_0,\ldots,x_{n+1})$ and $l(x_0,\ldots,x_{n+1})f_{d-1}(x_0, \ldots, x_{n+1})$ of degrees $1$ and $d$, respectively. The result follows by counting: there are 
		$(n+2)$ coefficients in the equation of $l$ and $\binom{n+d}{d-1}$ coefficients in the equation of $f_{d-1}$. Finally, we subtract two degrees of freedom because we can divide each polynomial by a non-zero coefficient. We obtain
		$$\dim(Z_1)=
		\binom{n+d}{d-1}+n
		$$
		which implies the formula for $\mathrm{codim}(Z_1)$ from observing that
		$$\dim(H_{n,d}\times H_{n,1})=\binom{n+1+d}{d}+n$$
		and by using the identity
		\begin{align}\label{eq:bin}
		\binom{n}{k}=\binom{n-1}{k-1}+\binom{n-1}{k}.
		\end{align}		
		To obtain the lower bound, we use the hypothesis that $n+1 \geqslant d$. First, we suppose that $n \geqslant d$. Observe that $\frac{n+d-k}{d-k}\geqslant \frac{n+d}{d}$ for all $k\geqslant 0$. Therefore, we obtain the following inequality:
		\begin{align*}
			\binom{n+d}{d}=\frac{(n+d)(n+d-1)\cdots(n+1)}{d!}=\frac{n+d}{d}\cdot \frac{n+d-1}{d-1}\cdots \frac{n+1}{1}\geqslant \left(\frac{n+d}{d}\right)^d.
		\end{align*}
		We conclude that 
		\begin{align}\label{eq:lowb}
		\binom{n+d}{d} \geqslant \left( \frac{n+d}{d} \right)^d \geqslant \left( \frac{n}{d}+1 \right)^2 \geq \left(\frac{n}{n+1}+1\right)^2\geqslant 2.
		\end{align}
		Finally, we suppose that $n+1=d$. Since $d\geqslant 2$, by \eqref{eq:bin} we get:
		\begin{align*}
		\binom{n+d}{d}=\binom{2d-1}{d} = \binom{2d-2}{d-1}+\binom{2d-2}{d}
		=\binom{2d-3}{d-2} +\binom{2d-3}{(d-1)} +\binom{2d-2}{d}\geqslant 3.
	\end{align*}
	\end{proof}
	\begin{cor}
	\label{corollary:codimension-reducible-hypersurfaces}
	For $n,d \geq 2$, let
	$$Z_1'=\{p\in H_{n,d}\ : \ p=l(x_0:\cdots:x_{n+1})\cdot f_{d-1}(x_0:\cdots:x_{n+1}),\ \deg(l)=1\}.$$
	Then $\mathrm{codim}(Z_1')=\binom{n+d}{d}-(n+2) \geqslant 2$.
\end{cor}
\begin{proof}
	 By counting the coefficients in the polynomials we find $(n+2)$ coefficients in the equation of $l$ and $\binom{n+d}{d-1}$ coefficients in the equation of $f_{d-1}$. We subtract two degree of freedom due to projective equivalence to obtain
	$$
	\dim(Z_1')= \binom{n+d}{d-1}+n.
	$$
	Therefore, the codimension is equal to an expression which we denote as $f(n,d)$.
	$$
	\mathrm{codim}(Z_1') =
	\left( \binom{n+d+1}{d}-1 \right)
	- 
	\left( 
	\binom{n+d}{d-1}+n
	\right)
	=
	\binom{n+d}{d} 
	- 
	\left( 
	n+1
	\right)
	:=
	f(n,d).
	$$
	To obtain the lower bound, we use \eqref{eq:bin} to obtain
	$$
	f(n,d)=f(n-1,d)+\binom{n+d-1}{d-1}-1.
	$$
	which implies $f(n,d) \geq f(n-1,d)$.
	If we set $n=2$, then we obtain
	\begin{align*}
	f(2,d):=  \binom{d+2}{d}-3=\frac{(d+1)(d+2)}{2}-3 
	\end{align*}
	which is an parabola with $f(2,2)=3$ and $f(2,d)\geqslant 3$ for $d>2$. The result follows.

\end{proof}

	\begin{lemma}\label{lemma:Z2-dim}
		For $n\geq 2$, let 
		$$
		Z_2=\{ (p,l)\in H_{n,d}\times H_{n,1}\ | \ X=\{p=0\},\ H=\{l=0\},\ \exists H'\neq H,\ \text {such that } X\cap H=X\cap H' \} .
		$$ 
		Then 
		$$\mathrm{codim}(Z_2)= \binom{n+d+1}{d} - \binom{n+d-1}{d-2}+n-2 \geq 2.$$
		
	\end{lemma}
	\begin{proof}
		Let $(p,l)\in Z_2$,  $X=\{p=0\}$, $H=\{l=0\}$, and let $H'\neq H$ be a hyperplane such that $X\cap H=X\cap H'$.
%		First, we claim that $X$ is the union of $d$ hyperplanes concurrent along $H \cap H'$.
		Without loss of generality, we may assume that $H=\{x_0=0\}$, $H'=\{x_1=0\}$. Since $X \cap H'=X \cap H$ is a Cartier divisor supported on $\{x_0=x_1=0\}\cong \mathbb P^{n-1}$ and $X \cap H$ is a hypersurface of degree $d$ in $H=\{x_0=0\}\cong \mathbb P^n$, we have that $X\cap H\subset H$ is $dL$, where $L=\{x_0=x_1=0\}\subset H$ is an $(n-1)$-dimensional hyperplane of $H$. Hence 
		\begin{align}\label{eq:XH}
		p=ax_1^d+bx_0f_{d-1}(x_0,\ldots, x_{n+1}).
		\end{align}
		Similarly, by considering $X \cap H'$ the equation of $X$ can be written as 
		\begin{align}\label{eq:XH'}
		p=a'x_0^d+b'x_1g_{d-1}(x_0,\ldots, x_{n+1}).
		\end{align}
		By comparing equations \eqref{eq:XH} and \eqref{eq:XH'}, we conclude that the equation of $(X,H)\in Z_2$ must have, up to isomorphism, the form $p=ax_0^d+a'x_1^d+x_0x_1g_{d-2}(x_0,\ldots, x_{n+1})$, $l(x_0,x_1)$.
		The formula for $\mathrm{codim}(Z_2)$  follows by counting.  The number of coefficients in the polynomial $g_{d-2}$ is 
$\binom{n+d-1}{d-2}$. The equation of $H$ is a linear combination 
of $x_0$ and $x_1$.  We must subtract two degrees of freedom because we can divide each polynomial by a 
non-zero coefficient. We obtain	
$$
\dim(Z_2)=\left( \binom{n+d-1}{d-2}+2 \right) +2 -2=\binom{n+d-1}{d-2}+2 
$$
which implies 		
\begin{align*}
\mathrm{codim}(Z_2)
&=
\binom{n+d+1}{d}+n - \left( \binom{n+d-1}{d-2}+2 \right)
=
\binom{n+d+1}{d} - \binom{n+d-1}{d-2}+n-2.
\end{align*}
The equality
$$
\binom{n+d}{d-1}=\binom{n+d-1}{d-2} + \binom{n+d-1}{d-1}
$$
leads us to $ - \binom{n+d-1}{d-2} \geq - \binom{n+d}{d-1}$.   Together with Lemma \ref{lemma:Z1-dim} we conclude:
$$
\mathrm{codim}(Z_2)
\geq
\binom{n+d+1}{d} - \binom{n+d}{d-1} +n-2
=
\binom{n+d}{d}+n-2\geqslant \binom{n+d}{d}\geqslant 2.
$$
\end{proof}
	
	\begin{lemma}\label{lemma:stabZ1Z2}
		Suppose $d\leqslant n+1$ and $0<t<t_{n,d}$. Then, for any $0<t<t_{n,d}$, the points in $Z_1\cup Z_2$ are GIT$_t$-unstable. In particular
		$$
		(H_{n,d}\times H_{n,1}  \setminus (Z_1\cup Z_2) )^{ss}_t=(H_{n,d}\times H_{n,1})^{ss}_t.
		$$
		
	\end{lemma}
	\begin{proof}
		By Theorem \ref{theorem:CY-Fano}, we know that pairs in $Z_1$ are unstable.   Let $(X,H)$ be a pair parametrised by 
		$(p,l)\in Z_2$. Then we can choose a coordinate system such that $p=x_0^d+x_1^d+x_0x_1p_{d-2}(x_0,x_1)$,  $l(x_0,x_1)$.
		
		The one-parameter subgroup $\lambda\colon \mathbb G_m\rightarrow G$, given by $\lambda(s)=\mathrm{Diag}(s^n,s^n, s^{-2}, \ldots, s^{-2})$ destabilises the pairs $(X,H)$ for any $ t \in (0,t_{n,d})$, by the Hilbert-Mumford criterion, since the Hilbert-Mumford function \cite[Lemma 2.2]{Gallardo-JMG-framework} gives $\mu_t(P,l, \lambda) = n+tn >0$.
		
	\end{proof}
	Define $\mathcal U\coloneqq (H_{n,d}\times H_{n,1})\setminus (Z_1\cup Z_2)$. Let $\pi_{V_{n,d}}\colon\mathcal V_{n,d}\rightarrow H_{n,d}$ be the universal family of hypersurfaces of degree $d$ in $\mathbb P^{n+1}$ and let $p_1\colon H_{n,d}\times H_{n,1}\rightarrow H_{n,d}$ and $p_2\colon H_{n,d}\times H_{n,1}\rightarrow H_{n,1}$  be the projections onto the first and second factor, respectively. We have that
	$$
	\mathcal V_{n,d}=\bigg\{(x_0,\ldots,x_{n+1}) \times a_I\in \mathbb P^{n+1}\times H_{n,d} \ | \ \sum a_Ix^I=0\bigg\}
	$$
	where $I$ runs over all partitions of $d$ of $n+2$ non-negative integers.
	
	Consider the following commutative diagram of fibre products
	\begin{equation*}
	\xymatrix{
		{\mathcal X} \ar[d]^\pi \ar[r] &\mathcal V_{n,d}\times H_{n,1}\ar[r] \ar[d]^{\pi_{\mathcal V_{n,d}}\times \mathrm{Id}_{H_{n,1}}} & \mathcal V_{n,d}\ar[d]^{\pi_{\mathcal V_{n,d}}}\\
		{\mathcal U}\ar@{^{(}->}[r]^(.25){i} & {H_{n,d}\times H_{n,1}} \ar[r]^(.6){p_1} \ar[d]^{p_2}& {H_{n,d}}\\
		{} & H_{n,1} & {},}
	\end{equation*}
	where the first square diagram is the fibre product of $i$ and $\pi_{\mathcal V_{n,d}}\times \mathrm{Id}_{H_{n,1}}$. We define $\pi\colon\mathcal X\rightarrow \mathcal U$ to be given by the fibre product in the first diagram, where $i\colon \mathcal U\hookrightarrow H_{n,d}\times H_{n,1}$ is the natural embedding.
	
	Observe that $\pi_{\mathcal V_{n,d}}\colon \mathcal V_{n,d}\rightarrow H_{n,d}$ is flat and proper because it is a universal family and $\pi_{\mathcal V_{n,d}}\times\mathrm{Id}_{H_{n,1}}$ is flat and proper because it is a pullback of a flat proper morphism. For the same reason, $\pi$ is flat and proper.
	
	Moreover $\mathcal U$ parametrises ---up to scaling by constants--- pairs $(p,l)$ of homogeneous polynomials in $n+2$ variables of degrees $d$ and $1$, respectively.
	
	Observe that
	\begin{align}\label{eq:UniverCur}
	\mathcal X=\{(x_0,\ldots,x_{n+1})\times a_I\times (b_0,\ldots, b_{n+1})\in \mathbb P^{n+1}\times \mathcal U \ | \ \sum a_Ix^I=0\}.
	\end{align}
	Define
	\begin{align}\label{eq:UnivDiv}
	\mathcal D=\{(x_0,\ldots,x_{n+1})\times a_I\times(b_0,\ldots,b_{n+1})\in \mathcal X \ | \ \sum b_ix_i=0\}\subset \mathcal X.
	\end{align}
	Clearly $\mathcal D$ is a Cartier divisor of $\mathcal X$ and the restriction of $\pi$ to $\mathcal D$, namely $\pi_{\mathcal D}\colon \mathcal D\rightarrow \mathcal U$ is an equidimensional morphism with a smooth base. Therefore $\pi_{\mathcal D}$ is flat and proper.
	
	\begin{lemma}
		\label{lemma:picard-isom}
		$$\mathrm{Pic}(\mathcal U)\cong \mathrm{Pic}(H_{n,d}\times H_{n,1})\cong p_1^*\left(\mathrm{Pic}\left(H_{n,d}\right)\right)\oplus p_2^*\left(\mathrm{Pic}\left(H_{n,1}\right)\right) \cong \mathbb Z^2,$$
		and if we identify $\mathcal L\in \mathrm{Pic}(\mathcal U)$ as
		$$\mathcal L\cong \mathcal O_{\mathcal U}(a,b)\coloneqq i^*\left(\mathcal O_{H_{n,d}}\left(a\right)\boxtimes \mathcal O_{H_{n,1}}\left(b\right)\right)|_{\mathcal U},$$
		via the above morphisms, then the line bundle $\mathcal L$ is ample if and only if $a>0$ and $b>0$.
	\end{lemma}
	\begin{proof}
		The codimension of $\mathcal U$ in $H_{n,d}\times H_{n,1}$ is at least $2$, by lemmas \ref{lemma:Z1-dim} and \ref{lemma:Z2-dim}. The result follows from \cite[Proposition II.6.5b]{hartshorneAG}.
	\end{proof}
	
	Notice that if $d\leqslant n+1$, we have that $-K_{\mathcal X/\mathcal U}$ is a relatively very ample line bundle. It follows from Lemma \ref{lemma:picard-isom} and Theorem \ref{theorem:CY-Fano} that $\Lambda_{CM, \beta}(-K_{\mathcal X/\mathcal U})\cong \mathcal O(a,b)$. Hence, by Lemma \ref{lemma:picard-isom}, we can extend this line bundle uniquely to the whole of $H_{n,d}\times H_{n,1}$. By a slight abuse of notation we will denote that extension by $-K_{\mathcal X/\mathcal U}$.
	Let $\beta\in (0,1)\cap \mathbb Q$. Define:
	$$\Lambda_{CM, \beta}\coloneqq\Lambda_{CM, \beta}(-K_{\mathcal X/\mathcal U})\coloneqq\Lambda_{CM, \beta}(\mathcal X, \mathcal D,-K_{\mathcal X/\mathcal U}).$$

	\begin{theorem}
		\label{theorem:git-cm-correspondence}
		Let $d\leqslant n+1$, $\beta\in (0,1]\cap \mathbb Q$. Then $\Lambda_{CM, \beta}(-K_{\mathcal X/\mathcal U})\cong \mathcal O(a(\beta),b(\beta))$. where
		\begin{align*}
		a(\beta)&=(n+2-d)^{n-1}\Bigg[(n+2-d)\Big((n+2)(d-1)(1+n(1-\beta))+(1-\beta)(n+1)\Big),\\
		&\qquad\qquad\qquad\qquad\qquad\qquad-nd(1-\beta)(n+1)\Bigg]>0, \\
		b(\beta)&=(n+2-d)^n d(n+1)(1-\beta)>0,\\
		\end{align*}
		and hence $\Lambda_{CM, \beta}$ is ample. 
		
		In particular, for $d=n+1$, and $t=\frac{b}{a}$, we have that:
		\begin{align*}
		t(\beta)=\frac{d^2(1 - \beta)} {
			d^2-\beta
		}
		\end{align*}
	\end{theorem}
	\begin{proof}
		This theorem is a direct consequence of applying Theorem \ref{theorem:CM-chern-class} to two particular type of pencils. To determine the value of $a$, we use a pencil of log pairs where the hypersurfaces vary while the hyperplane section is fixed. To determine the value of $b$, we use a pencil of log pairs where the hypersurface is fixed while the hyperplane sections vary.  Next, we consider each case in detail.
		
		Let $X_1$ and $X_2$ be two different very general hypersurfaces of degree $d$ in $\PP^{n+1}$ and let its base locus be $C:=X_1 \cap X_2$. By Corollary \ref{corollary:codimension-reducible-hypersurfaces}, we may assume that all the elements in the pencil are irreducible and hence the intersection of a fibre with an element of the pencil determines a hyperplane section.
		
		Then,  we construct a pencil 
		of hypersurfaces of degree $d$ in $\PP^{n+1}$ by taking 
		$Bl_{C}\PP^{n+1} \to \PP^1$. 
		We denote this pencil as $\mathcal Y/\mathbb \PP^1\subset \mathrm{Bl}_C\mathbb P^{n+1}/\mathbb{P}^1$ and let $\pi_{\mathcal Y}\colon \mathcal Y \rightarrow \mathbb P^1$ be the natural projection. Notice that we have a commutative diagram
		\begin{equation*}
		\xymatrix{
			{\mathcal Y} \ar[d]^{\pi} \ar@{^{(}->}[r]^(.5){i} &\mathcal X\ar[d]^(.25){\pi} & &\\
			{\mathbb P^1}\ar@{^{(}->}[r]^(.5){i} &   {\mathcal U} \ar@{^{(}->}[r]^(.25){j}& {H_{n,d}\times H_{n,1}}  \ar[d]^{\pi_{1}} \ar[r]^(.5){\pi_{2}} & H_{n,1}\cong \mathbb P^{n+1} \\
			&									&	 H_{n,d}\cong\mathbb P^N. &
		}
		\end{equation*}
		
		Let $p_{H_0}$ be a point in $H_{n,1}$ parametrizing a general hyperplane $H_0$.
		Then \eqref{eq:UniverCur} implies that 
		$$
		\mathcal Y \times p_{H_0} \subset \PP^1 \times \PP^{n+1} \times H_{n,1}.
		$$
		In particular, $\mathcal Y$  is a hypersurface of bidgree $(1,d)$ in $\PP^1 \times \PP^{n+1} $.
		We are interested in constructing a family $(\mathcal Y, \mathcal D)$ of log pairs formed by a fibre of $\pi_{\mathcal Y}$ and the divisor obtained by intersecting each fibre with $H_0$. Hence
		$$\mathcal{D} = \{r\in \mathbb P^1\times\mathbb P^{n+1} \ | \ r\in \mathcal Y \text{ and } r|_{\mathbb P^{n+1}}\in H_0\}$$
		is the complete intersection of two hypersurfaces of degree $(1,d)$ and $(0,1)$ in $\PP^1 \times \PP^{n+1}$. Let $\pi_{\mathbb P^1}\colon \mathbb P^1\times\mathbb P^{n+1}\rightarrow \mathbb P^{1}$ and $\pi_{\mathbb P^{n+1}}\colon \mathbb P^1\times\mathbb P^{n+1}\rightarrow \mathbb P^{n+1}$ be the natural projections.  Let $H_{\mathbb P^{n+1}}=\pi^*_{\mathbb P^{n+1}}(\mathcal O_{\mathbb P^{n+1}}(1))$ and $H_{\mathbb P^{1}}=\pi^*_{\mathbb P^{1}}(\mathcal O_{\mathbb P^{1}}(1))$. By adjunction
		\begin{align*}
		K_{\mathcal Y} &= \left( K_{\PP^1 \times \PP^{n+1}} + Bl_{C}\PP^{n+1} \right)
		|_{\mathcal Y }
		\\
		&=
		\left( -2H_{\PP^1}-(n+2)H_{\PP^{n+1}} + H_{\PP^1}+dH_{\PP^{n+1}} \right)
		|_{ \mathcal Y } 
		\\
		&=
		\left( -H_{\PP^1}-(n+2-d)H_{\PP^{n+1}} \right)|_{ \mathcal Y }.
		\end{align*}
		from which we deduce 
		\begin{align*}
		K_{\mathcal Y / \mathbb P^1} 
		&= K_{\mathcal Y}  - \pi^* K_{\mathbb P^1}   
		\\
		&=\left( -H_{\PP^1}|_{\mathcal Y}-(n+2-d)H_{\PP^{n+1}}|_{\mathcal Y }\right)+2H_{\PP^1}|_{\mathcal Y}
		\\
		&=H_{\PP^1}|_{\mathcal Y}-(n+2-d)H_{\PP^{n+1}}|_{\mathcal Y },
		\end{align*}
		The expression
		\begin{align*}
		(- K_{\mathcal Y / \mathbb P^1} )^n 
		&=
		\left( -H_{\PP^1}|_{\mathcal Y}+(n+2-d)^{n-1}H_{\PP^{n+1}} |_{\mathcal Y } \right)^n
		\\
		&=
		-n(n+2-d)H_{\PP^1}|_{\mathcal Y}\cdot H_{\PP^{n+1}}^{n-1} |_{\mathcal Y }+
		(n+2-d)^nH_{\PP^{n+1}}^n |_{\mathcal Y }
		\end{align*}
		implies that $c_1(- K_{\mathcal Y / \mathbb P^1} )^n \cdot \mathcal D $ is equal to
		\begin{align*}
		\left(
		-n(n+2-d)^{n-1}H_{\PP^1}\cdot H_{\PP^{n+1}}^{n-1} +
		(n+2-d)^nH_{\PP^{n+1}}^n 
		\right)
		\cdot
		(H_{\PP^1} +dH_{\PP^{n+1}})  \cdot H_{\PP^{n+1}}
		\end{align*}
		from which we obtain
		\begin{align*}
		c_1(- K_{\mathcal Y / \mathbb P^1} )^n \cdot \mathcal D
		&= 
		\left( (n+2-d)^n-nd(n+2-d)^{n-1} \right)
		H_{\PP^{1}} H^{n+1}_{\PP^{n+1}}.
		\end{align*}
		Similarly, we obtain 
		\begin{align*}
		(- K_{\mathcal Y / \mathbb P^1} )^{n+1}
		=
		-(n+1)(n+2-d)^nH_{\PP^{1}}\cdot H^{n}_{\PP^{n+1}}|_{\mathcal Y}+
		(n+2-d)^{n+1}H^{n+1}_{\PP^{n+1}}|_{\mathcal Y},
		\end{align*}
		which implies
		\begin{align*}
		c_1(- K_{\mathcal Y / \mathbb P^1} )^{n+1}
		&=
		(-(n+1)(n+2-d)^nH_{\PP^{1}}\cdot  H^{n}_{\PP^{n+1}}+
		(n+2-d)^{n+1}H^{n+1}_{\PP^{n+1}})|_{\mathcal Y}
		\\
		&=
		\left(
		-(n+1)(n+2-d)^nH_{\PP^{1}} \cdot H^{n}_{\PP^{n+1}}+
		(n+2-d)^{n+1}H^{n+1}_{\PP^{n+1}}
		\right)
		\cdot
		\left(
		H_{\PP^{1}} +dH_{\PP^{n+1}}
		\right)
		\\
		&=
		\left(
		-d(n+1)(n+2-d)^n+
		(n+2-d)^{n+1}
		\right)
		H_{\PP^{1}}\cdot  H^{n+1}_{\PP^{n+1}}
		\\
		&=
		(n+2-d)^n
		(n+2)(1-d)
		H_{\PP^{1}} \cdot H^{n+1}_{\PP^{n+1}}.
		\end{align*}
		By Theorem \ref{theorem:CM-chern-class}, we deduce that
		\begin{equation}
		\label{eq:computing-a}
		\deg(i^*\Lambda_{CM,\beta}) =-(1+n(1-\beta)) \pi_*\left(c_1\left(-K_{\mathcal Y/\mathcal B}\right)^{n+1} \right)
		+( 1-\beta ) \left(n+1\right) \pi_* \left( c_1 \left(-K_{\mathcal Y/\mathcal B}\right)^{n} 
		\cdot \mathcal D  
		\right).
		\end{equation}
		Observe that 
		$$\deg((j\circ i)^*(\Lambda_{CM,\beta}))=\deg((j\circ i)^*\circ\pi_1^*(\mathcal O_{\mathbb P^{N}}(a)))=a.$$
		Hence, substituting the values in \eqref{eq:computing-a}, we obtain the value of $a$.
		
		Next,  we calculate $b$.
		For that purpose, we consider a pencil of pairs induced by a general smooth hypersurface $S\subset \mathbb P^{n+1}$, represented by $p_S\in H_{n,d}$  and a pencil of  hyperplanes  $H(t)$  with $t \in \PP^1$. 
		We claim that 
		$$
		\mathcal D|_{p_S \times \PP^1}
		=
		(f_S=f_{H(t)}=0)
		\subset
		p_S \times \mathbb{P}^3 \times \mathbb{P}^1
		$$
		where $f_S$ and $f_{H(t)}$ are the equations of the hypersurface $S$ and the pencil of hyperplanes.
		$\mathcal D|_{p_S \times \PP^1}$ is a complete intersection in $\PP^{n+1} \times \PP^1$ of two hypersurfaces of bi-degree  $(d,0)$ and $(1,1)$. Because $S\times \mathbb P^1/\mathbb P^1$ is a trivial fibration, we have that $c_1(- K_{S\times \mathbb P^1 / \mathbb P^1} )^{n+1}=0$.
		
		Then 
		\begin{align*}
		K_{(S \times \PP^1) / \PP^1} = K_{S \times \PP^1} - \pi^* K_{\PP^1} 
		=  
		K_{S} \otimes \mathcal O_{\PP^1}
		=
		(d-n-2)H_{\mathbb P^{n+1}}|_S \otimes \mathcal O_{\PP^1}
		\end{align*} 
		which implies
		\begin{align*}
		\deg(\Lambda_{CM,\beta}) 
		&=
		-(1-\beta)(n+1)\left( K_{(S \times \PP^1) / \PP^1}  \right)^n \cdot \mathcal D
		\\
		&=-
		(n+1)(1 - \beta)(d-n-2)^n\left(  H_{\PP^{n+1}} \right)^n \cdot dH_{\PP^{n+1}}  
		\cdot (H_{\PP^{n+1}}+H_{\PP^{1}}) 
		\\
		&=
		-(n+1)(1 - \beta)(d-n-2)^n\left(  H_{\PP^{n+1}} \right)^n 
		\cdot (dH_{\PP^{n+1}}H_{\PP^{1}}) 
		\\
		&=
		-d(n+1)(1 - \beta)(d-n-2)^n\left(  H_{\PP^{n+1}}^{n+1}\cdot H_{\PP^{1}}  \right)^n
		\end{align*}
		This implies that 
		$$
		b= (d-n-2)^nd(n+1)(1 - \beta))
		$$
	\end{proof}

\begin{rmk}
	\label{remark:special-OPS}
		Cubic surfaces with only $A_1$ (or $A_2$) singularities play an important role in the classical GIT quotient first considered by Hilbert, since they are precisely the (semi-)stable ones. Not surprisingly, they also play a similar role in the GIT quotient of log pairs in the first polarisation chamber, with little added complexity provided by the boundary. In Proposition \ref{prop:sss_A2} below we summarise their GIT$_t$ stability for all values of $t$.
	
	In addition, for those log pairs which are GIT$_t$-unstable, we provide a one-parameter subgroup $\lambda_*$ such that the limit of the log pair (which in principle is only a class representing a cubic polynomial and a monomial) is also a log pair. That such a one-parameter subgroup exists is not at all obvious. Let us illustrate this. Recall that the scheme $H_{n,d}\times H_{n,1}$ represents pairs of homogeneous polynomials, not necessarily log pairs, e.g. all pairs $(lq_2,l)$ (where $q_2$ is a homogeneous polynomial of degree $2$) do not represent a log pair. In principle, for a given GIT$_t$-unstable pair $(S,D)$ defined by polynomials $(p,l)$, it could happen that for all one-parameter subgroups $\lambda$ which \emph{destabilise} $(S,D)$, the pair $(\overline p, \overline l)=\lim_{t \to 0}(p,l)$ does not represent a log pair or, equivalently, $\overline l$ divides $\overline p$. If we can find a one-parameter subgroup $\lambda_*$ such that $(\overline p, \overline l)$ represents a log pair then we can use its orbit to induce a destabilising test configuration, as we do later in Theorem \ref{theorem:K-stability-implies-GIT}. In fact, this is precisely what we use Proposition  \ref{prop:sss_A2} for in the proof of Theorem \ref{theorem:K-stability-implies-GIT} (i). We should remark that this is the price to pay when cooking up a GIT quotient for log pairs which is easy to characterise geometrically.
	
	One may ask why we do not include a similar statement to Proposition for all other GIT$_t$-unstable pairs. Probably we can find a similar statement for cubic surfaces whose singularities are worse than $A_2$ (at least for most of them). However the proof would have to be specific to each case and since we do not need such a general result, we do not include it. 
	
\end{rmk}
\begin{prop}
		\label{prop:sss_A2}
		Let $S$ be a cubic surface which is either smooth or has $\boldsymbol{A}_1$ or $\boldsymbol{A}_2$ singularities, $D\in |-K_S|$ and $t\in (0,1)$. The pair $(S, D)$ is 
		\begin{enumerate}[(i)]
			\item GIT$_t$-stable if and only if the support of $D$ does not contain any surface singularity of $S$ of type $\boldsymbol{A}_2$,
			\item GIT$_t$-polystable but not GIT$_t$-stable if and only if $S=S_*$ is the unique surface with three $\boldsymbol{A}_2$ singularities and $D=D_*$ is the unique hyperplane section in $S_*$ consisting of three lines, each passing through two of the singularities,
			\item GIT$_t$-semistable but not GIT$_t$-polystable if and only if $S$ has at least an $\boldsymbol{A}_2$ singularity $p$ such that $p\in \mathrm{Supp}(D)$ and $D$ has normal crossings at $p$ (i.e. $D$ is reduced and has an $\boldsymbol{A}_1$ singularity at $p$),
			\item GIT$_t$-unstable if and only if $S$ has at least one singularity $p$ of type $\boldsymbol{A}_2$ such that $D$ is either non-reduced or has a cuspidal singularity at $p$.
		\end{enumerate}
	Moreover, there is a destabilising one-parameter subgroup $\lambda_*$ such that if $(S,D)$ is GIT$_t$-unstable the natural morphism $\pi\colon \overline{\lambda_* \cdot S }\subset \mathbb P^3\times\mathbb P^1\rightarrow \mathbb P^1$ has an irreducible fibre at $0\in \mathbb P^1$. Furthermore, if $(S,D)$ is GIT$_t$-semistable but not GIT$_t$-polystable, then $\lim_{t \to 0}\lambda\cdot (S,D)=(S_*,D_*)$.
	\end{prop}
	\begin{proof}
		Parts (i) and (ii) follow from \cite[Theorems 1.3 and 1.4]{Gallardo-JMG-cubic-surfaces}. Hence, we may assume that $S$ has at least one $\boldsymbol{A}_2$ singularity. By \cite[p. 255]{classificationcubics} $S$ is irreducible. Let $S$ be defined by the equation $f=\sum f_Ix^I$, where $f_I\in \mathbb C$ and $I\in \mathbb Z^4$ runs over all partitions of $3$ of four non-negative integers. Let $D=S\cap H$, where $H$ is defined by the equation $h=\sum_{i=0}^3h_ix_i$ with $h_i\in \mathbb C$.

		Given a one-parameter subgroup $\lambda\colon \mathbb G_m\rightarrow G\coloneqq \mathrm{SL}(4,\mathbb C)$, we say that $\lambda$ is \emph{normalised} if $\lambda=\mathrm{Diag}(s^a)$, where $a\in \mathbb Z^4$, $a=(a_0,a_1,a_2,a_3)$, $a_0\geqslant a_1\geqslant a_2\geqslant a_3$, $\sum_{i=0}^3 a_i=0$. We have a natural pairing between normalised one-parameter subgroups and monomials $x^I=x_0^{d_0}x_1^{d_1}x_2^{d_2}x_3d^3$ given by $\langle x^I,\lambda\rangle=\sum_{i=0}^{3}a_id_i$. Given a normalised one-parameter subgroup $\lambda=\mathrm{Diag}(s^a)$, $a=(a_0,a_1,a_2,a_3)$, and if $j=\sup\{0,1,2,3|h_i\neq 0\}$, we define the functions
		\begin{equation}
		\label{eq:HM-function-def}
		\langle f, \lambda\rangle \min\{\langle x^I,\lambda\rangle \ : \ f_I\neq 0\}, \qquad \langle h, \lambda\rangle = \min \{a_i\ : \ h_i\neq 0\}=a_j.
		\end{equation}
	There is a natural order of monomials of fixed degree (sometimes known as \emph{Mukai's order} \cite{Gallardo-JMG-framework}). Namely, $x^I\leqslant x^J$ if and only if for each one-parameter subgroup $\lambda$, $\langle x^I,\lambda\rangle\leqslant \langle x^J,\lambda\rangle$. Therefore, it is enough to test the pairing $\langle f, \lambda\rangle$ for the set of minimal monomials with non-zero coefficients in $f$.
	
	By \cite[p. 255]{classificationcubics}, if $S$ has an $\boldsymbol{A}_2$ singularity, then it has at most two other singularities, which are $\boldsymbol{A}_1$ or $\boldsymbol{A}_2$ singularities. Moreover, by \cite[Lemma 3]{classificationcubics} (c.f. \cite[Lemma 4.2]{Gallardo-JMG-cubic-surfaces}), $S$ has a $\boldsymbol{A}_2$ singularity at a point $p\in\mathrm{Supp}(D)$ if and only if $(S,D=S\cap H)$ is conjugate by an element of $\mathrm{Aut}(\mathbb P^3)$ to a pair defined by equations
	\begin{align}
		f&=x_0x_1x_3+x_2^3+x_2^2f_1(x_0:x_1)+x_2f_2(x_0:x_1)+f_3(x_0:x_1),
		\label{eq:a2-equations}\\
		h&=h_0x_0+h_1x_1+h_2x_2,\nonumber
	\end{align}
	where $f_d$ are homogeneous polynomials of degree $d$ and $p$ is conjugate to $(0:0:0:1)$. For the rest of the proof we will therefore assume that $(S,D)$ is given by \eqref{eq:a2-equations} and that $S$ is singular at $p\coloneqq(0:0:0:1)\in \mathrm{Supp}(D)$. The Mukai order of monomials of degree $3$ in four variables is described in \cite[Figure 7.3]{mukai-book-moduli}, from which we can deduce that the minimal monomials which have non-zero coefficient in $f$ are $x_0x_1x_3$ and $x_2^3$. Hence, if $\lambda =\mathrm{Diag}(s^a)$, $a=(a_0,a_1,a_2,a_3)$, then
	\begin{equation}
	\label{eq:a2-HM-function}
		\langle f, \lambda\rangle = \min\{a_0+a_1+a_3, 3a_2\}.
	\end{equation}
	
	In \cite{Gallardo-JMG-framework} a finite set of one-parameter subgroups $S_{n,d}$ was introduced. This set determines the GIT$_t$-stability of any pair $(X, D=X\cap H)$ of dimension $n$ and degree $d$. In \cite[Lemma 2.1]{Gallardo-JMG-cubic-surfaces}, the set $S_{2,3}$ was computed. More explictly, if $(S,D)$ is defined by $f$ and $h$ and $t\in (0,1)$, the pair $(S, D)$ is GIT$_t$-semistable if and only if
	$$\max_{\lambda\in S_{2,3}}\mu_t(f, h, \lambda)\coloneqq \max_{\lambda\in S_{2,3}}\{\langle f,\lambda\rangle + t \langle h, \lambda \rangle \}\leqslant 0$$
	and $(S, D)$ is not GIT$_t$-stable if in addition there is a $\lambda_i\in S_{2,3}$ such that $\mu_t(S,D,\lambda_i)=0$. The set $S_{2,3}$ includes the element $\lambda_*=\mathrm{Diag}(S^{a_*})$, where $a_*=(1,1,0,-2)$ ($\lambda_*=\overline \lambda_2$ in the notation of \cite[Lemma 2.1]{Gallardo-JMG-cubic-surfaces}).
	
	Suppose that $h_2\neq 0$ in \eqref{eq:a2-equations}. Then, running \eqref{eq:HM-function-def} and \eqref{eq:a2-HM-function} through the finite list of elements $\lambda\in S_{2,3}$, gives that $\mu_t(S, D, \lambda)\leqslant 0$ for all $t\in (0,1)$ and all $\lambda \in S_{2,3}$. Moreover $\mu_t(S, D, \lambda_*)=0$. Hence $(S, D)$ is strictly GIT$_t$-semistable.	Conversely, suppose that $h_2=0$. Then $\mu_t(S,D,\lambda_*)=t>0$ for $t\in (0,1)$ and $(S,D)$ is GIT$_t$-unstable. Hence, $(S,D)$ with one surface singularity of type $\boldsymbol{A}_2$ at $p\in \mathrm{Supp}(D)$ is GIT$_t$-unstable if and only if $h_2=0$.
	
	Moreover, notice that $\lim_{t\rightarrow 0}\lambda_*(t)\cdot f\eqqcolon f_0=x_0x_1x_3+x_2^3$, which is the equation for the unique cubic surface $S_*$ with three $\boldsymbol{A}_2$ singularities (and therefore irreducible). If in addition $h_2\neq 0$, then $\lim_{t\rightarrow 0}\lambda_*(t)\cdot h=x_2$, which is the hyperplane section in $S_*$ corresponding to $D_*$, the union of the unique three lines in $S_*$, each passing through two of the singularities of $S_*$.
	
	The only thing that remains to show is that $h_2\neq 0$ and $h_2=0$ in \eqref{eq:a2-equations} characterise $(S,D)$ as the log pairs described in (iii) and (iv) in the statement, respectively. To see this, notice that if $h_2\neq 0$, then we may describe $D$ by substituting $x_2=-\frac{h_0}{h_2}x_0-\frac{h_1}{h_2}x_1$ and $x_3=1$ in the equation of $f$ in \eqref{eq:a2-equations} to obtain the equation of $D$ in $H=\{h=0\}$ localised at $p$ as $x_0x_1+g_3(x_0:x_1)$, where $g_3$ is a homogeneous polynomial of degree $3$. Hence $D$ has a nodal singularity at $p=(0:0:0:1)$ and by the classification of singularities of plane cubic curves (see, e.g. \cite[Table 2]{Gallardo-JMG-cubic-surfaces}), $D$ has only normal crossing singularities. 
	
	Conversely, if $h_2=0$, then either $h_1\neq 0$ or $h=x_0$. In the former case, substituting $x_1=cx_0$, $c=-h_0/h_1\in \mathbb C$ in the equation of $f$ in \eqref{eq:a2-equations} and localising at $p$ we obtain that the equation of $D$ is locally given by $cx_0^2+x_2^3+g_3(x_0:x_2)$ and hence $D$ has either a cuspidal singularity at $p$ ($c\neq 0$) or $D$ is the union of three lines (counted with multiplicity) intersecting at $p$. Finally, if $h=x_0$, then substituting $x_0=0,x_3=1$ on $f$ we obtain that the equation of $D$ is the union of three lines through $p$, counted with multiplicity.
	\end{proof}
	\begin{theorem}
		\label{theorem:K-stability-implies-GIT}
		Let $(X,D)$ be a log pair of dimension $n$, $L$ be a very ample line bundle of $X$, such that $i\colon X\hookrightarrow \mathbb P^N$ (such that $L=i^*\left(\mathcal O_{\mathbb P^N}\left(1\right)\right)$),  and $\beta\in (0,1)$, where either:
		\begin{enumerate}[(i)]
			\item $X\subset \mathbb P^3$ is a cubic surface which is smooth or has only $\boldsymbol{A}_1$ and $\boldsymbol{A}_2$ singularities, with the embedding $i$ giving the natural inclusion, $D=X\cap H$ is a hyperplane section, $N=3$, $Y\coloneqq\mathbb P(H^0(\mathbb P^{3}, \mathcal O(3)))\times \mathbb P(H^0(\mathbb P^{3}, \mathcal O(1)))$, $G=\mathrm{SL}(4, \mathbb C)$ acts naturally on $Y$, or
			\item $X$ is a normal Fano variety, $D\in \mathbb P(H^0(-aK_X))\eqqcolon Y$ is some non-trivial Cartier divisor for some $a\in \mathbb N$, $G=\mathrm{Aut}(X)$ is reductive and its character group is trivial and acts naturally on $Y$ with its natural representation. Suppose further that there is some GIT polystable $D_0\in \mathbb P(H^0(-aK_X))$, such that $(X, (1-\beta)D_0)$ is log K-polystable.
		\end{enumerate}
		Let $\Lambda_{CM,\beta}$ be the log CM line bundle of the universal family $\pi\colon \mathcal Y\rightarrow Y$. We have that $\Lambda_{CM,\beta}$ is ample. Suppose $(X,(1-\beta)D,L)$ is log K-(semi/poly)stable, then for each of the two situations described above we have:
		\begin{enumerate}[(i)]
			\item $(X, D)$ is GIT$_{t(\beta)}$ (semi/poly)stable, where $t(\beta)=\frac{b(\beta)}{a(\beta)}$ for $b,a$ as in Theorem \ref{theorem:git-cm-correspondence},
			\item $D$ is GIT (semi/poly)stable,
		\end{enumerate}
		with respect to the $\Lambda_{CM,\beta}$-polarisation of $Y$.
	\end{theorem}
	\begin{proof}
		Note that by naturality of the construction $\Lambda_{CM,\beta}$ is $G$-linearised. Furthermore, $L$ is $G$-linearised (see, for instance \cite[Lemma 2.1]{Gallardo-JMG-framework}) and both linearisations are the same up to rescaling since the $G$-linearised Picard group injects into the Picard group as the character group of $G$ is trivial. Firstly, we prove the ampleness of the log CM line bundle. In case (i), we have that $\Lambda_{CM,\beta}\cong \mathcal O(a,b)$ for some $a,b\in \mathbb Z_{>0}$, by Theorem \ref{theorem:git-cm-correspondence}, and hence  $\Lambda_{CM, \beta}$ is ample.
		
		In case (ii) we have that $\mathrm{rk}(\mathrm{Pic}(Y))=1$, so $\Lambda_{CM, \beta}$ is either ample, antiample or trivial. In the antiample and trivial cases, we may take any non-trivial test configuration $(\mathcal X, \mathcal D, \mathcal L)$ for $(X, (1-\beta)D_0)$ where both $L$ and $\Lambda_{CM, \beta}$ are defined and agree for all fibres, including the central fibre (e.g. by taking any non-trivial one-parameter subgroup which kills some monomials of $D_0$ as $t$ goes to $0$). On the one hand $w(\Lambda_{CM, \beta}(\mathcal X, \mathcal D, \mathcal L))\leqslant 0$ since $\Lambda_{CM,\beta}^{-1}$ is nef. On the other hand, $\mathrm{DF}_\beta(\mathcal X, \mathcal D, \mathcal L)>0$ since $(X,(1-\beta)D_0)$ is log K-polystable and $(\mathcal X, \mathcal D, \mathcal L)$ is not trivial, contradicting Theorem \ref{theorem:weight-DF}.
		
		The proof follows the same strategy as in \cite[Theorem 3.4]{Odaka-Spotti-Sun}. Let $\lambda$ be a one-parameter subgroup of $G$ acting on a point $p\in Y$, which represents either a log pair $(X, H)$ where $X$ is a cubic surface, $H\not\subset X$ is a hyperplane and $D=X\cap H$ is a hyperplane section (case (i)) or a pluri-anticanonical section $D$ (case (ii)). Consider the natural projection $\pi\colon\overline{\mathcal X} \coloneqq\overline{\lambda\cdot p}\subset Y\times \mathbb P^1\rightarrow \mathbb P^1$ and by abuse of notation $\pi\colon\mathcal X=\overline{\mathcal X}\setminus \{\pi^{-1}(\infty)\}\rightarrow \mathbb C$. Let $q=\pi^{-1}(0)\in Y$ be the central fibre of $\pi$. In case (ii), $q$ is a hypersurface $\overline D$ and hence $\overline X$ is a test-configuration. In case (i), $q\in Y$ is a pair $(\overline X, \overline H)$ formed by a cubic surface $X$ and a hyperplane $H$. If $\overline H\not\subset \overline X$ (for instance, if $\overline X$ is irreducible), then $\overline D\coloneqq \overline H\cap \overline X\in |-K_{\overline X}|$ and $\lambda$ induces a test configuration of $(X,D)$. By Proposition \ref{prop:sss_A2}, if $X$ is smooth or it has only $\boldsymbol{A}_1$ or $\boldsymbol{A}_2$ singularities and it is not semistable, we can find a destabilising one-parameter subgroup that induces a destabilising test configuration. Hence, in both case (i) and (ii) if $(X,(1-\beta)D,L)$ is log K-(semi)stable, the statement follows at once from Theorem \ref{theorem:weight-DF}, which allow us to write the  Hilbert-Mumford numerical criterion in terms of the Donaldson-Futaki invariant of the induced test configuration. 
		
		Now suppose that $(X, (1-\beta)D, L)$ is  K-polystable, then it is also K-semistable. We distinguish the two cases separately:
		\begin{enumerate}[(i)]
			\item Let $p\in Y$ be the point representing $(X, H)$. 
			\item Let $p\in Y$ be the point representing $D$.
		\end{enumerate} 
		The point $p$ is GIT semistable. If it is not GIT polystable, by definition there is a one-parameter subgroup $\gamma(t)$ of $G$ such that the pair $\overline p=\lim_{t\rightarrow 0}\gamma(t)\cdot p$ is GIT polystable but not GIT stable. Let $(\mathcal X, \mathcal D, \mathcal L)$ be the test configuration induced by $\gamma$ and let $(\overline X, \overline D)$ be its central fibre (note that the induced one parameter subgroup is in $Aut(\overline X, \overline D)$). Then $w(\Lambda_{CM,\beta}(\mathcal X, \mathcal D, \mathcal L))=0$. Theorem \ref{theorem:weight-DF} implies $\mathrm{DF}_\beta(\mathcal X, \mathcal D, \mathcal L)=0$, but as $(X, (1-\beta)D)$ is log K-polystable, $(\mathcal X,\mathcal D)\cong (X\times \mathbb C, D\times \mathbb C)$, once the $\mathbb C^*$-action is ignored. Hence $(\overline X, \overline D)\cong (X,D)$. But then $\overline p = p$ is GIT polystable.
	\end{proof}
	
	\begin{rmk}
		To our knowledge this is the first article in the literature using the moduli continuity method to describe compactifications of log K-stable pairs. Therefore, we find it is important to stress the obstruction we encountered when compactifying the orbit of one-parameter subgroups to obtain test configurations in the proof of Theorem \ref{theorem:K-stability-implies-GIT}, since this problem does not appear in the case of cubic surfaces (without boundary) or more generally in the hypersurfaces cases of \cite{Odaka-Spotti-Sun}. See Remark \ref{remark:special-OPS} for such an account. We must stress that this difficulty does appear when dealing with other varieties (no boundary) such as complete intersections \cite{Odaka-Spotti-Sun, Spotti-Sun-delPezzo-quadrics} and a different workaround to address it in the case of quadrics can also be provided by analysing the GIT quotient, see Appendix of  \cite{Spotti-Sun-delPezzo-quadrics}. A more desirable solution than the one we use would involve showing that any GIT$_t$-unstable pair $(X,D)$ can be destabilised by a one-parameter subgroup whose limit is a not too singular log pair $(\overline X, \overline D)$. However, we are not able to prove this in full generality, but in a case-by-case basis.
	\end{rmk}
	
	\section{Proofs of the main theorems}
	\label{sec:proofs}
	
	We first observe that in all the cases considered in our main theorems we know that for all $\beta\in (0,1)$ there exists conical K\"ahler-Einstein metrics on log pairs $(X,(1-\beta)D)$ for $X$ and $D$ smooth (pluri)anticanonical, which have positive Einstein constant if  $\beta$ is greater than the log Calabi-Yau threshold. This follows by the interpolation property of log $K$-stability \cite{Li-Sun-conical-KE}, since $X$ is K\"ahler-Einstein, and K\"ahler-Einstein metrics exist for all small enough value of $\beta$. Now we restrict to the case when $(X,(1-\beta)D)$ is a log smooth log Fano pair. By the work of Chen-Donaldson-Sun \cite{Chen-Donaldson-Sun-Kstability-all}, if we have a sequence $(X_i,(1-\beta)D_i, g_i)$ of singular K\"ahler-Einstein metrics $g_i$ on $X_i$ with conical singularities of angle $2\pi\beta$ along $D_i$ where the $X_i$ are deformation equivalent, we can take  subsequences converging in the Gromov-Hausdorff topology  to $(W,(1-\beta)\Delta, g_\infty)$, a klt weak K\"ahler-Einstein Fano pair, where $W$ is a $\mathbb{Q}$-Gorenstein smoothable Fano variety with the same degree as  $X_i$.
	
	Our next goal is to obtain some  a-priori control on the singularities appearing in the Gromov-Hausdorff limits. The crucial estimate we need is the following purely algebro-geometric result of Li and Liu \cite[Proposition 4.6]{Li-Liu-volume-minimization},  which states the following: 
	
	\begin{theorem}[{\cite[Proposition 4.6]{Li-Liu-volume-minimization}}]
		\label{theorem:liu-estimate}
		For any K-semistable log Fano pair $(W,(1-\beta)\Delta)$, 
		$$ (-K_W-(1-\beta)\Delta)^n \leq \Big(1+\frac{1}{n}\Big)^n \widehat{\mathrm{vol}}_{(W,(1-\beta)\Delta), p},$$
		for $p$ any point in $W$. 
	\end{theorem}
	
	Here  $\widehat{\mathrm{vol}}_{(W,(1-\beta)\Delta), p}$ is the normalised volume at the singularity $p\in W$, defined as the infimum of local volumes of valuations centred at $p$ normalised by their log discrepancies,  as in \cite{Li-Liu-Xu-normalized-volume-guide}. It follows from the definition that the normalised volume satisfies
	\begin{equation}
	\label{eq:volume-inequality-simplification}
	\widehat{\mathrm{vol}}_{(W,(1-\beta)\Delta), p}\leq \widehat{\mathrm{vol}}_{W,p}, \; \mbox{for any} \; \beta \in (1-\mbox{lct}(W,\Delta),1),
	\end{equation}
	where $\mbox{lct}(W,\Delta)$ denotes the log canonical threshold of the pair $(W,\Delta)$.
	One can easily verify \eqref{eq:volume-inequality-simplification} for the normalised volume of any quasi-monomial valuations (and as a result, for all valuations). 
	
	Since by Berman \cite{Berman-KE-implies-K-polystability}, a K\"ahler-Einstein log Fano pair is log K-polystable, we can use the above estimate to control the singularities of the limit ambient space $W$. Let us analyze the cases that appear in our main theorems.
	
	\begin{prop}
		\label{prop:GH-limit-cubic pairs} Let $(W,(1-\beta)\Delta)$ be the Gromov-Hausdorff limit of smooth K\"ahler-Einstein pairs $(X_i,(1-\beta)D_i)$ with $X_i$ smooth cubic  surface and $D_i$ smooth hyperplane section. For any $\beta>\beta_0=\frac{\sqrt{3}}{2}$, the surface $W$ must be itself defined as cubic surface in $\mathbb{P}^3$ whose singular locus (if non-empty) consists of $\boldsymbol{A}_1$ or $\boldsymbol{A}_2$ singularities and the limit divisor $\Delta$ is a hyperplane section. 
	\end{prop}
	\begin{proof}		
		The degree of the limit pair  is 
		$$(-K_W-(1-\beta)\Delta)^2=(-K_{X_i}-(1-\beta)D_i)^2=(-\beta K_{X_i})^2=3\beta^2,$$
		by continuity of volumes. Moreover, note that $W$ must have only isolated quotient singularities locally analytically isomorphic to $\mathbb{C}^2/\Gamma$,  where $\Gamma$ is a finite subgroup of $U(2)$ acting freely on $S^3$, since $W$ must have at worst klt singularities, and klt surface singularities are precisely quotient singularities \cite[6.11]{Klemens-Kollar-Mori}. Moreover, by Liu \cite{Liu-volume-bound-surfaces} the local normalised volume for quotient singularities is simply given by $\widehat{\mathrm{vol}}_{\mathbb{C}^2/ \Gamma ,0}= \frac{4}{\vert  \Gamma \vert }$.  Thus, from Theorem \ref{theorem:liu-estimate} and \eqref{eq:volume-inequality-simplification}, it follows that
		$$\frac{4}{\vert  \Gamma \vert }=\widehat{\mathrm{vol}}_{W,p}\geq \frac{4}{3} \beta^2.$$
		If we take $\beta>\beta_0=\frac{\sqrt{3}}{2}$, then $\vert  \Gamma \vert \leq 3$. By the classification of $\mathbb{Q}$-Gorenstein smoothable surface singularities \cite[Proposition 3.10]{kollar-shepherd-barron},  $\Gamma$ must be a cyclic group acting in $SU(2)$. Hence the singularities of $W$ are canonical and, by the classification of del Pezzo surfaces with canonical singularities, $W$ is a cubic surface with at worst $A_1$ or $A_2$ singularities. In particular, any anticanonical section of $W$ is given by a hyperplane section, as claimed.
	\end{proof} 
		
	For the case of $\mathbb{P}^n$ we need to assume the following conjecture.
	
		\begin{conjecture}[{Gap Conjecture \cite[Conjecture 5.5, c.f. Conjecture 2.1]{Spotti-Sun-delPezzo-quadrics}}]
		\label{conjecture:gap}	
			Let $p$ be a klt singularity of an $n$-dimensional variety $W$. Then $$\widehat{\mathrm{vol}}_{W,p}\leq 2(n-1)^n.$$
		\end{conjecture}
		We remark that this is know to hold in dimension two by Liu \cite{Liu-volume-bound-surfaces}, and in dimension three by Liu and Xu \cite[Theorem 3.1.1]{Liu-Xu-Kstability-cubic-threefolds}. 
		
		\begin{prop}
		\label{prop:GH-Pn-hypersurfaces}	
		Let $(W,(1-\beta)\Delta)$ be  the Gromov-Hausdorff limit of smooth K\"ahler-Einstein pairs $(X_i,(1-\beta)D_i)$ with each $X_i$ isomorphic to $\mathbb{P}^n$  and $D_i$ smooth hypersurfaces of degree $d$. If we choose $\beta$ such that $1>\beta>\beta_0$, for $\beta_0$ as in \eqref{eq:beta0-Pn}, 
		then $W$ is also isomorphic to $\mathbb{P}^n$ and $\Delta$ is a hypersurface of degree $d$, \emph{provided} that the  Gap Conjecture \ref{conjecture:gap} holds.
		\end{prop}
		\begin{proof}

		Arguing as in the proof of Proposition \ref{prop:GH-limit-cubic pairs}, we have that $$(-K_W-(1-\beta)\Delta)^n=(n+1-(1-\beta)d)^n,$$
		and it is then easy to see that by taking $1>\beta> \beta_0= 1-\Big(\frac{n+1}{d}\big(1-2^{1/n}(1-\frac{1}{n})\Big)$, the normalised volume of any eventual singularity $p$ of the limit space $W$ must have volume
		$$\widehat{\mathrm{vol}}_{W,p}>2(n-1)^n.$$
		This contradicts the formula in Conjecture \ref{conjecture:gap}, which we assume to hold. Hence, $W$ is smooth. Moreover, by Chen-Donaldson-Sun \cite{Chen-Donaldson-Sun-Kstability-all} we know that $W$ is deformation equivalent to $\mathbb{P}^n$. Since the Fano index is constant along smooth deformations \cite[Proposition 6.2]{Gounelas-invariants-Fano-families}, the Fano index of $W$ is $n+1$. Hence $W$ is biholomorphic to $\mathbb P^n$ by Kobayashi-Ochiai \cite{Kobayashi-Ochiai-characterizations-CPn-hyperquadrics}, and the claims follow.
		\end{proof}
		
		\begin{rmk}
			Notice that there are know examples of smooth toric Fanos having the same degree as the projective space  starting from dimension five (the unique projective toric smooth Fano $5$-fold $X$ with $(-K_X)^5=(-K_{\mathbb P^5})^5=7776$ whose toric polytope $Q$ has $8$ vertices, $18$ facets and $\mathrm{vol}{Q}=18$, \cite{Obro-smooth-Fano-polytopes-algorithm, GRDB}). Thus, just having the same volume is not sufficient to determine that $W\cong \mathbb P^n$ and we really use that the index is preserved.
		\end{rmk}
		
		In order to conclude the proof of the main theorems one runs the moduli continuity method, which has already appeared in the literature for the absolute case (see, e.g. \cite{Odaka-Spotti-Sun}). For the reader's convenience, we briefly recall the method:		
		\begin{proof}[{Proof of Theorems \ref{theorem:main-cubicsurfaces} and \ref{theorem:main-Pn}}]
			First of all, notice that thanks to Theorem \ref{theorem:K-stability-implies-GIT}, propositions \ref{prop:GH-limit-cubic pairs}, \ref{prop:GH-Pn-hypersurfaces}, and  \cite{Berman-KE-implies-K-polystability}, we can define a natural map:
			$$\phi: \overline{M}_{\beta}^{GH} \longrightarrow \overline{M}^{GIT_\beta},$$
			where:\textsl{}
			\begin{enumerate}[(i)]
				\item in the situation described in the statement of Theorem \ref{theorem:main-cubicsurfaces}, $1>\beta>\beta_0=\frac{\sqrt{3}}{2}$, $\overline{M}_{\beta}^{GH}$ denotes the Gromov-Hausdorff compactification of the log pairs $(X,(1-\beta) D)$ formed by a del Pezzo surface $X$ of degree $3$ admitting a singular K\"ahler--Einstein metric with conical singularities of angle $2\pi\beta$ along an anticanonical section $D$, and $ \overline{M}^{GIT_\beta}$ denotes the  explicit GIT$_{t(\beta)}$ stability quotient of log pairs $(X,D)$, where $t(\beta)=\frac{9(1-\beta)}{9-\beta}$;
				\item in the situation described in the statement of Theorem \ref{theorem:main-Pn}, $1>\beta>\beta_0$ for $\beta_0$ as in \eqref{eq:beta0-Pn}, $\overline{M}_{\beta}^{GH}$ denotes the Gromov-Hausdorff compactification of the log pairs $(X,(1-\beta) D)$, where $X=\mathbb P^n$ admits a K\"ahler--Einstein metric with conical singularities of angle $2\pi\beta$ along $D\subset\mathbb P^n$ is a hypersurface of degree $d$, and $ \overline{M}^{GIT_\beta}$ denotes the  explicit GIT stability quotient of hypersurfaces $D\subset\mathbb P^n$ of degree $d$.
			\end{enumerate}
		Since singular K\"ahler-Einstein metrics with conical singularities of fixed angle are unique \cite{Chen-Donaldson-Sun-Kstability-all}, the map $\phi$ is injective. In addition, $\phi$ is a continuous map between the Gromov-Hausdorff topology in $\overline{M}_{\beta}^{GH}$ and the euclidean topology of $ \overline{M}^{GIT_\beta}$ . The latter follows by \cite{Chen-Donaldson-Sun-Kstability-all} and the Luna slice theorem (see also \cite{Spotti-Sun-Yao-singular-KE-Kstability} and \cite{Li-Wang-Xu-compact-moduli}). The image of $\phi$ is open and dense since all smooth log pairs are dense and K\"ahler-Einstein, and compact since $\overline{M}_{\beta}^{GH}$ is compact. Thus $\phi$ must be surjective. Since $\phi$ is a continuous map between a compact space and a Hausdorff space, it must be a homeomorphism. The claim for the line bundle in the statement of Theorem \ref{theorem:main-cubicsurfaces} follows from Theorem \ref{theorem:git-cm-correspondence}.
		\end{proof}

		\begin{rmk}
			We expect that similar analysis may be performed for other examples of Fano pairs (e.g., of cubic threefolds with a hyperplane section, compare \cite{Gallardo-JMG-cubic-surfaces, Liu-Xu-Kstability-cubic-threefolds}). 
		\end{rmk}

		\subsection{Explicit examples of K-polystable pairs}
		
		Thanks to the explicit description of GIT$_t$-stability for log pairs consisting of a cubic surface and a hyperplane section given in \cite[Theorems 1.3, 1.4]{Gallardo-JMG-cubic-surfaces}, we can \emph{characterise} precisely which of these log pairs are log K-polystable (equivalently which of these log pairs admit a K\"ahler-Einstein metric with conical singularities) when $\beta$ is large enough.	
		\begin{prop}\label{prop:K-stability-cubics-chamber1} For  $\beta>\beta_0=\frac{\sqrt{3}}{2}$, a log pair $(S,(1-\beta) D)$ (consisting of a del Pezzo surface $S$ of degree three and an anticanonical divisor $D\in |-K_S|$) is log K-polystable if and only if it belongs to the following list:		
			\begin{enumerate}[(i)]
				\item $S$ has finitely many singularities of types at worst $\boldsymbol{A}_1$ or $\boldsymbol{A}_2$ and if $P\in D$ is a surface singularity, then $P$ is at worst an $\boldsymbol{A}_1$ singularity of $S$.
				\item $S$ is the unique cubic surface with three $\boldsymbol{A}_2$ singularities and $D$ is the divisor consisting of the union of the unique three lines in $S$, each  of them containing two of the two $\boldsymbol{A}_2$ singularities.
			\end{enumerate}
		\end{prop}
		\begin{proof}
			From theorems \ref{theorem:main-cubicsurfaces} and \ref{theorem:git-cm-correspondence} we know that $(S,(1-\beta) D)$ is log K-polystable if and only if it is GIT$_t$-stable for
			$$0<t(\beta)=\frac{9(1-\beta)}{9-\beta}<t(\beta_0)=\frac{3}{107} (33 - 16 \sqrt 3)\approx 0.148<\frac{1}{5}$$
			when $\beta>\beta_0$. The description of the GIT$_t$-polystable pairs for $t\in(0,\frac{1}{5})\cap \mathbb Q$ follows from \cite[Theorems 1.3, 1.4]{Gallardo-JMG-cubic-surfaces}.
		\end{proof}

		Similarly we now apply Theorem \ref{theorem:main-Pn} to a survey of the GIT stability of hypersurfaces in low degrees and dimension to deduce K-polystability of pairs  $(\mathbb{P}^n,(1-\beta)H_d)$:
		
		\begin{prop}  Let $(n,d)$ in
			$$ \Omega\coloneqq\Big\{(2,3), (2,4), (2,5),(2,6),(3,3), (3,4)\Big\}$$
			and $1>\beta>\beta_0$ where $\beta_0$ is as in \eqref{eq:beta0-Pn}. A log pair $(\mathbb{P}^n,(1-\beta)D)$ (where $D\subset \mathbb P^n$ is a hypersurface of degree $d$) is log K-polystable if and only if the corresponding following conditions hold for each choice of $(n,d)\in \Omega$:
			\begin{itemize} 
				\item[$(2,3)$] $D$ is smooth or the union of three non-concurrent lines;
				\item[$(2,4)$] $D$ is either:
				\begin{enumerate}[(a)]
					\item reduced and smooth or has $\boldsymbol{A}_1$ and $\boldsymbol{A}_2$ singularities,
					\item $D=2C$ where $C$ is a smooth conic;
				\end{enumerate}
				\item[$(2,5)$] $D$ is smooth or has isolated singularities of type $\boldsymbol{A}_k$, $\boldsymbol{D}_4$ or $\boldsymbol{D}_5$, or there is some $(a:b)\in \mathbb P^1$ such that $D$ is isomorphic to
				$$\{x_1(x_0^2x_2^2+2ax_0x_2x_1^2+bx_1^4)=0\};$$
				\item[$(2,6)$] $D$ is a sextic curve satisfying one of the following:
				\begin{enumerate}
					\item $D$ is reduced and has ADE singularities,
					\item $D$ is the union of three distinct conics,
					\item $D$ is the union of a double line and an irreducible quartic curve,
					\item $D=2C_1+C_2$ where $C_1, C_2$ are conics, $C_1$ is irreducible and $C_1$ and $C_2$ intersect with simple normal crossings,
					\item $D=2C$, where $C$ is a smooth cubic,
					\item $D=2C_1+C_2$ where $C_1$ and $C_2$ are two conics tangent to each other at two points,
					\item $D=3(L_1+L_2+L_3)$, where $L_1,L_2,L_3$ are three distinct lines with no common intersection;
				\end{enumerate}
				\item[$(3,3)$] $D$ has at worst finitely many $\boldsymbol{A}_1$ singularities or $D$ has precisely three $\boldsymbol{A}_2$ singularities;
				\item[$(3,4)$] $D$ is a quartic surface satisfying one of the following:
				\begin{enumerate}
					\item $D$ is smooth or has ADE singularities of type $\boldsymbol{A}_n, \boldsymbol{D}_n, \boldsymbol{E}_6, \boldsymbol{E}_7, \boldsymbol{E}_8$,
					\item $D$ has a double point $P$ of type $\widetilde {\boldsymbol{E}}_8$ and some ADE singularities of type $\boldsymbol{A}_n$, $\boldsymbol{D}_n$, $\boldsymbol{E}_6$, $\boldsymbol{E}_7$, $\boldsymbol{E}_8$ such that no line in $D$ contains $P$,
					\item $D$ is singular along an ordinary nodal curve $C$ and some ADE singularities of type $\boldsymbol{A}_n$, $\boldsymbol{D}_n$, $\boldsymbol{E}_6$, $\boldsymbol{E}_7$, $\boldsymbol{E}_8$. Either $D$ is irreducible and $C$ is a nonsingular curve of degree $2$ or $3$ with four simple pinch points (i.e. locally analytically isomorphic to $x_1^2+x_2^2x_3=0$) or $D$ consists of two quadric surfaces which intersect transversely along a nonsingular elliptic curve of degree $4$,
					\item the singular locus of $D$ consists of a double point $P$  locally analytically isomorphic to $x_1x_2x_3+x_1^2+x_2^3=0$ and some ADE singularities of type $\boldsymbol{A}_n, \boldsymbol{D}_n, \boldsymbol{E}_6, \boldsymbol{E}_7, \boldsymbol{E}_8$ such that no line in $D$ contains $P$,
					\item either $D$ is singular along a strictly quasi-ordinary nodal curve $C$ and some ADE singularities of type $\boldsymbol{A}_n, \boldsymbol{D}_n, \boldsymbol{E}_6, \boldsymbol{E}_7, \boldsymbol{E}_8$, such that no line in $D$ contains a double pinch point (i.e. locally analytically isomorphic to $x_1^2+x_2^2g=0$ where $g\in \mathbb C[[x_2,x_3]], g=x_3^2 \mod x_2$), or $C$ is a nonsingular rational curve of degree $2$ and $D$ has either two double pinch points on $C$ or one double pinch point and two simple pinch points on $C$,
					\item either the singular locus of $D$ consists of two double points of type $\widetilde {\boldsymbol{E}}_8$ or it consists of two double points of type $\widetilde {\boldsymbol{E}}_7$ and some ADE singularities of type $\boldsymbol{A}_n, \boldsymbol{D}_n, \boldsymbol{E}_6, \boldsymbol{E}_7, \boldsymbol{E}_8$,
					\item the singular locus of $D$ consists of two skew lines, each of which is an ordinary nodal curve with four
					simple pinch points,
					\item $D$ consists of a plane and a cone over a nonsingular cubic curve in the plane,
					\item the singular locus of $D$ consists of a nonsingular, rational curve $C$ of degree 2 or 3, and some
					ADE singularities of type $\boldsymbol{A}_n$, $\boldsymbol{D}_n$, $\boldsymbol{E}_6$, $\boldsymbol{E}_7$, $\boldsymbol{E}_8$,  $C$ is a strictly quasi-ordinary (i.e. $D$ has no points on $C$ of multiplicity larger than $2$ and each pinch point on $C$ is a double pinch point), nodal curve and the set of pinch points consists of two double pinch points such that each double pinch point lies on a line in $D$,
					\item $D$ consists of two, nonsingular, quadric surfaces which intersect in a reduced curve $C$ of arithmetic genus $1$. $C$ consists of two or four lines such that its singularities consist of $2$ or $4$ ordinary double points; the dual graph of $C$ is homeomorphic to a circle,
					\item $D$ consists of four planes with normal crossings.
				\end{enumerate}
			\end{itemize}
				In addition, let $1>\beta>\frac{3+8\sqrt[3]{2}}{15}$ and $D\subset \mathbb{P}^3$ be one of the following quintic surfaces:
				\begin{enumerate}
					\item $D$ is either a smooth quintic surface or has only isolated singularities of type ADE,
					\item $D$ has only isolated double point singularities and isolated triple point singularities with reduced tangent cone,
					\item $D$ is a normal surface such that each of its singularities has either Milnor number smaller than 22 or modality smaller than 5,
					\item $D$ has isolated minimal elliptic singularities, 
					\item $D$ is the union of a double smooth quadric surface and a transverse hyperplane,
					\item $D$ is a generic quintic surface with a curve of singularities of multiplicity three such that the 	support of that curve does not contain any lines,
					\item $D$ is irreducible with a curve of singularities supported on a reduced curve $C$ such that the genus of $C$ is greater than one, $C$ does not contain any lines, and $D$ does not have an additional line of singularities.
				\end{enumerate}
		Then $(\mathbb P^3, (1-\beta)D)$ is log K-polystable.			
		\end{prop}
		\begin{proof}
			This is a straight-forward application of Theorem \ref{theorem:main-Pn} and the classification of GIT polystable hypersurfaces. For $(n, d)=(2,3), (2,4), (3,3)$ see \cite[Examples 7.12, 7.13 and theorems 7.14 and 7.24]{mukai-book-moduli} and \cite[page 80]{MumfordGIT}. For $(n,d)=(2,5)$ see \cite[page 80]{MumfordGIT} and \cite[Corollary 5.14 and page 149]{Laza-thesis}. For $(n,d)=(2,6)$, see \cite[Theorem 2.4]{Shah-GIT-sextic-plane-curves}. For $(n,d)=(3,4)$ see \cite[Theorem 2.4]{Shah-GIT-sextic-plane-curves}. For $(n,d)=(3,5)$ see \cite{Gallardo-GIT-quintics}.
		\end{proof}

		\subsection{Wall-crossing}\label{wall}
		
		It is natural to ask what happens in the previous examples when the value of $\beta$ becomes smaller. The natural expectation is that birational modifications are occurring, giving rise to wall-crossing phenomena in the space of log K-stability conditions.  This paradigm should be thought as an analogue of variations of GIT stability for the log CM line bundle. There are two approaches when considering this analogy.
		
		For the fist approach we must recall that it has been conjectured by Odaka-Spotti-Sun that the CM line bundle should be ample in the good moduli (in the sense of Alper \cite{Alper-good-moduli}) of families of K-polystable Fano varieties \cite{Odaka-Spotti-Sun}.

		It is natural to expect a similar conjecture for the log CM line bundle over families of log K-polystable pairs to hold. However, in this case the positivity of the CM line bundle is dependent on the choice of K-stability condition $\beta$. If the log version of the Odaka-Spotti-Sun conjecture holds, as $\beta$ is perturbed  the variation of the positivity of the CM line bundle over families would detect wall-crossing phenomena in the compact moduli space.
		
			\begin{rmk}
				The ampleness of the log CM line bundle has been recently established in the pre-print \cite{XZ}.
			\end{rmk}
			
			 In variations of GIT quotients these wall-crossing modifications tessellate the space of GIT stability conditions into ``chambers'' and ``walls'' (see \cite{Gallardo-JMG-cubic-surfaces} for a full description of this tessellation in the case of cubic surfaces and hyperplane sections). The GIT quotient remains constant when the stability condition is perturbed within a chamber, while it undergoes a \emph{Thaddeus flip} \cite{Thaddeus-vGIT, Dolgachev-Hu-vGIT} when it is perturbed through a wall.
		
		The second approach to understand the analogy between variations of log K-stability and variations of GIT stability is with regards to particular examples: it is not surprising that in some cases the setting of variations of log K-stability cannot be reduced only to the study of classical variations of GIT quotients. For example, when we consider the log pair $(\mathbb{P}^2, (1-{\beta_4})H_4)$, it follows by the study of del Pezzo surfaces of degree $2$  \cite{Odaka-Spotti-Sun} that for $\beta=\frac{1}{2}$ the GIT picture above is not accurate any more. More precisely:
		\begin{prop}[{\cite{Odaka-Spotti-Sun}}] The K-polystable compactification  of the moduli of log pairs $(\mathbb{P}^2, (1-\beta) H_4)$ where $H_4$ is a plane quartic curve and $\beta=\frac{1}{2}$ is given by the blow-up of the GIT quotient of quartic plane curves at the point representing the double conic. Over this point the Gromov-Hausdorff limits are  log pairs of type $(\mathbb{P}(1,1,4), \frac{1}{2}H)$, where $H$ is the hyperelliptic curve $z_3^2=f_8(z_1 ,z_2 )$, where $f_8$ is a  polystable binary octic. 
		\end{prop}
		
		De Borbon \cite{deBorbon-cubic-curves} has informally discussed the case  $(\mathbb{P}^2, (1-{\beta})H_3)$, where $H_3$ is a plane cubic curve. More generally, it seems very intriguing to fully understand the case of  $(\mathbb{P}^2, (1-{\beta})H_d)$, even when $H_d$ is a hypersurface whose degree $d$ is small as $\beta$ gets smaller. In fact, one can hope to be able to study such problem with the help of the Hacking-Prokhorov classification of $\mathbb{Q}$-Gorenstein degenerations of the projective plane \cite{Hacking-Prokhorov}. Any such degeneration is given by a weighted projective plane of the form $\mathbb{P}(a^2,b^2,c^2)$ satisfying a Markov type equation $a^2+b^2+c^2=3abc$. The log curve is going to degenerate to certain weighted hypersurfaces. Thus one should be able to  study the K-stability problem by ``gluing'' quotients as in \cite{Odaka-Spotti-Sun}.
		
		We should moreover stress that another situation in which wall-crossing phenomena for the K-polystable compact moduli space can be observed is given in Fujita's work on weighted hyperplane arrangements \cite{Fujita-hyperplane-arrangements}.
		
		Regarding the case of a cubic surface and a hyperplane section, we should remark that some classical wall-crossing picture is described in \cite{Gallardo-JMG-cubic-surfaces}. Our present estimate for $\beta$ is not strong enough to reach the first wall in the space of K-stability conditions. Hence our main theorem is unable to realise a wall-crossing phenomena. Moreover, we should remark that for values of $\beta$ smaller than our bound, one a-priori is forced to consider degenerations of the cubic surfaces to singular non-canonical del Pezzo surfaces of degree three, embedded in three-dimensional weighted projective spaces. These situations may be addressed through a more technical study of K-stability in such explicit situations. In particular, it is highly probable that for small $\beta$, the compact moduli of K-polystable pairs will not coincide with the one induced in GIT$_t$-stability. It would be very interesting to find whether this expectation is correct and what the natural replacement should be.
		
		\subsection{Further evidence supporting Conjecture \ref{conjecture:main}}
		In this section we make some further comment on Conjecture \ref{conjecture:main}. In particular we can check that the conjecture holds for the case of the unique del Pezzo surface $S$ of degree $3$ with three $\boldsymbol{A}_2$ singularities and $l=1$. This surface is given by $S=\{x_0x_2x_3=x_1^3\}\subset \mathbb P^3$, and it is toric. In order to see if the conjectural picture holds, we first need to discuss the geometry of the surface. The singular points of $S$ are located at:
		\begin{align}
	p_0=[1:0:0:0],  && p_2=[0:0:1:0],  & & p_3=[0:0:0:1].\label{eq:points_3A2}
	\end{align}
	$S$ has three lines $\overline{p_0p_2}$, $\overline{p_0p_3}$, and $\overline{p_2p_3}$ which join the corresponding singular points. There is an unique hyperplane section $H_*\coloneqq \{x_1=0\}|_S=\overline{p_0p_2}+\overline{p_0p_3}+\overline{p_2p_3}\sim -K_S$ that contains the three singularities. In addition, there are three hyperplanes $H_{23}=\{x_0=0\}$, $H_{03}=\{x_2=0\}$, $H_{02}=\{x_3=0\}$,  which induce non-reduced hyperplane sections $H_{ij}|_S=3\overline{p_ip_j}$.
	
	\begin{lemma}\label{lemma:classification-3A2-anticanonical} Let $\mathcal H_i \subset |-K_S|=|\mathcal O_S(1)|\cong (\mathbb P^3)^*$ be the $2$-dimensional linear system  (or net) whose base locus is the singular point $p_i$ and $j,k$ be distinct indices such that $\{i,j,k\}=\{0,2,3\}$. The net $\mathcal H_i$ contains three $1$-dimensional linear systems (or pencils) $\mathcal H_i^1, \mathcal H_i^2, \mathcal H_i^3\subset \mathcal H_i$ whose common intersection is empty and such that the following hold:
		\begin{enumerate}[(i)]
			\item the elements of $-K_S\setminus\left(\bigcup_{i=1}^3\mathcal H_i	\right)$ are the hyperplane sections which do not contain any of the points $p_0,p_2,p_3$,
			\item the elements of $\mathcal H_i\setminus (\mathcal H_i^1\cup \mathcal H_i^2\cup \mathcal H_i^3)$  are irreducible plane cubic curves with precisely one nodal ($\boldsymbol{A}_1$) singularity supported at $p_i$,
			\item $\mathcal H_i^2\cap \mathcal H_i^3=H_*|_S=\overline{p_0p_2}+\overline{p_0p_3}+\overline{p_2p_3}$, $\mathcal H_i^1\cap \mathcal H_i^2=3\overline{p_ip_j}$ and $\mathcal H_i^1\cap \mathcal H_i^3=3\overline{p_ip_k}$,
			\item the elements of $\mathcal H_i^2\setminus (\mathcal H_i^1\cup \mathcal H_i^3)=\mathcal H_i^2\setminus  \{H_*|_S\cup 3\overline{p_ip_j}\}$ (respectively, $\mathcal H_i^3\setminus (\mathcal H_i^1\cup \mathcal H_i^2)=\mathcal H_i^3\setminus  \{H_*|_S\cup 3\overline{p_ip_k}\}$) are reducible plane cubic curves $\overline{p_ip_j}+C$ where $C$ is an irreducible plane conic containing $p_i$ and $p_j$ and coplanar to $\overline{p_ip_j}$ (respectively, $\overline{p_ip_k}+C$ where $C$ is an irreducible plane conic containing $p_i$ and $p_k$ and coplanar to $\overline{p_ip_k}$),
			\item the elements of $\mathcal H_i^1\setminus \{3\overline{p_ip_j}\cup 3\overline{p_ip_j}\}$ are irreducible plane cubic curves intersecting $\overline{p_ip_j}$ and $\overline{p_ip_k}$ only at $p_i$, intersecting  $\overline{p_jp_k}$ away from $p_j$ and $p_k$ and with precisely one cuspidal ($\boldsymbol{A}_2$) singularity supported at $p_i$.
		\end{enumerate}
	\end{lemma}
	\begin{proof}
	Part (i) is automatic from \eqref{eq:points_3A2}. Without loss of generality, let $i=0$ (the other cases are symmetrical). We have that
	$$\mathcal H_0=\{(x_0:x_1:x_2:x_3)\in S\ : \ c_1x_1+c_2x_2+c_3x_3=0,\ (c_1:c_2:c_3)\in \mathbb P^2\}.$$
	The pencils $\mathcal H_0^1, \mathcal H_0^2,\mathcal H_0^3$ correspond to taking $c_1=0$, $c_2=0$ and $c_3=0$, respectively in the above description. From this description it is clear that (iii) holds. Let $H=\{c_1x_1+c_2x_2+c_3x_3=0\}|_S\in \mathcal H_0$. If $H$ is as in (ii), we may write $x_1=-\frac{c_2}{c_1}x_2-\frac{c_3}{c_1}x_3$ and equation of $H$ is given by
	$$x_0x_2x_3-\left(-\frac{c_2}{c_1}x_2-\frac{c_3}{c_1}x_3\right)^3=0$$
	which is easily seen to be irreducible with a nodal singularity at $(1:0:0)$ when $c_2,c_3\neq 0$, proving (ii). If $H\in \mathcal H_i^2\setminus (\mathcal H_i^1\cup \mathcal H_i^3)$, then $c_2=0$ and $c_1,c_3\neq 0$, so the equation of $H$ is
	$$x_0x_2x_3-\left( \frac{c_3}{c_1}x_3 \right)^3 = x_3 \left( x_0x_2- \left( \frac{c_3}{c_1}x_3 \right)^2 \right)=0,$$
	which corresponds to the situation described in (iv) (the case for $H\in \mathcal H_i^3\setminus (\mathcal H_i^1\cup \mathcal H_i^2)$ is analogous). Finally, if $H\in \mathcal H_i^1\setminus (\mathcal H_i^2\cup \mathcal H_i^3)$, then $c_1=0$, $c_2,c_3\neq 0$, $x_3=-\frac{c_2}{c_3}x_3$ and the equation of $H$ is $ \left(-\frac{c_2}{c_3}x_0x_2^2-x_1^3=0 \right)$, as in (v).
\end{proof}
	
	\begin{lemma}{\cite[Sec 4.9]{sakamaki2010automorphism}}
		\label{lemma:AutS}
		Let $\Sigma_3$ be the symmetric group of order $3$ acting naturally on the coordinates of $\mathbb P^3$ and 
		let $S$ be the cubic surface given by  $(x_0x_2x_3-x_1^3=0)$, then $\mathrm{Aut}(S)\cong(\mathbb C^*)^2 \rtimes \Sigma_3$.
	\end{lemma}
Observe that $\mathrm{Aut}(S)$ is reductive and there is a natural action of $\mathrm{Aut}(S)$ on
\begin{equation}
	\label{eq:cubic-sections-identification}
	\mathbb H\coloneqq \mathbb P(H^0(S, -K_S))\cong \mathbb P (H^0(S, \mathcal O_S(1)))\cong (\mathbb P^3)^*\cong \mathbb P^3.
\end{equation}
We will describe this action and determine the stability of the GIT quotient $(\mathbb P(H^0(S, -K_S)))^{ss}\git\mathrm{Aut}(S)$ in two steps. First we will describe the action of the neutral component $\mathrm{Aut}^0(S)\cong \mathbb (\mathbb C^*)^2$ and the stability of $(\mathbb P(H^0(S, -K_S)))^{ss}\git\mathrm{Aut}^0(S)$ and then we will extend this analysis to $(\mathbb P (H^0(S, -K_S)))^{ss}\git\mathrm{Aut}(S)$. We will use the description of polystable points to prove Conjecture \ref{conjecture:main} for $X=S$ and $l=1$ and then show how the description of the GIT quotient can be used in combination with the moduli continuity method to prove a stronger statement. 

Fix a set of coordinates and let $\{x_0,x_1,x_2,x_3\}$ be the natural basis of the space $H^0(S, -K_S)$, which is projectivised in \eqref{eq:cubic-sections-identification}. The action of the neutral component of the automorphism group $(\mathbb C^*)^2\triangleleft\mathrm{Aut}^0(S)$, with coordinates $(t_0, t_2)$, on $H^0(S, -K_S)$ is defined over the basis elements by
	\begin{align}
	(t_0,t_2) \cdot [x_0,x_1,x_2,x_3]=[t_0x_0, x_1, t_2x_2, t_0^{-1}t_2^{-1}x_3].\label{eq:torus-action}
	\end{align}
Next, we describe the GIT quotient of  $\mathbb P(H^0(S, -K_S)) \cong \mathbb P^3$ with respect to the action of $(\mathbb C^*)^2$ described in \eqref{eq:torus-action}. 
It is well known that such quotients are not unique, but they rather depend on the $(\mathbb C^*)^2$-linearisation of the line bundle $\mathcal O_{\mathbb H}(1)$.  We follow the framework of \cite[Sec 7]{dolgachev2003lectures}.  
By the arguments that precede \cite[Corollary 7.1 and Theorem 7.2]{dolgachev2003lectures}, we have the short exact sequence of groups
$$
0
\to
\chi \left( (\mathbb C^*)^2 \right) 
\to 
\text{Pic}^{(\mathbb C^*)^2} \left( \mathbb H \right)
\to
\text{Pic} \left( \mathbb H \right)
\to
0,
$$
where $\chi \left( (\mathbb C^*)^2 \right)=\mathrm{Hom}((\mathbb C^*)^2,\mathbb C^*)$ is the group of rational characters of $(\mathbb C^*)^2$. Since  $\chi \left( (\mathbb C^*)^2 \right)  \cong \mathbb Z^2$ and $\text{Pic} \left( \mathbb H \right)
\cong \mathbb Z$, we obtain 
$
\text{Pic}^{(\mathbb C^*)^2} \left( \mathbb P^3 \right)
\cong 
\mathbb Z^3
$.
A $(\mathbb C^*)^2$-linearised bundle must of the form $\mathcal O_{\mathbb H}(m)$, so the linearisation is given by a linear representation of $(\mathbb C^*)^2$ in $\left( \mathbb C[x_0, x_1,x_2, x_3]_m \right)^*$, the dual of $\mathbb C[x_0, x_1,x_2, x_3]_m $. Such action is defined by
the following formula (see \cite[Example 8.2]{dolgachev2003lectures}):
$$
(t_0, t_2) \cdot x_{i_1} \cdots x_{i_m} \to t_0^{-a'}t_2^{-b'}x_{i_1} \cdots x_{i_m},
$$
for some integers $a'$ and $b'$. Then, the $(\mathbb C^*)^2$-linearised bundle can be indexed by the triples
$(m,a',b') \in \mathbb Z^3$. The GIT quotient does not change if we rescale the parameters $(m,a',b')$.
Hence, we can set $m=1$ and the GIT quotient will depend of two 
rational parameters $(a,b):=\left( \frac{a'}{m}, \frac{b'}{m} \right)$ which define the following action of $(\mathbb C^*)^2$ 
on $\left( \mathbb C[x_0, x_1,x_2, x_3]_{m=1} \right)^*$.
\begin{equation}\label{eq:linear}
(t_0,t_2) \cdot [x_0,x_1,x_2,x_3]=[
t_0^{-a+1}t_2^{-b}x_0, 
t_0^{-a}t_2^{-b}x_1, 
t_0^{-a}t_2^{-b+1}t_2x_2, 
t_0^{-a-1}t_2^{-b-1}x_3].
\end{equation}
	\begin{lemma}\label{lemma:StaC*2}
		Given the action of $(\mathbb C^*)^2$ on $\mathbb P(H^0(S, -K_S))$,
		 in \eqref{eq:torus-action}, linearised by $a=b=0$ in \eqref{eq:linear}, a hyperplane section $H$ is
		\begin{enumerate}[(i)]
			\item
			stable if and only if it does not contain any of the $\boldsymbol{A}_2$ singularities of $S$, i.e. $H\in |-K_S|\setminus\bigcup_{i=1}^3 \mathcal H_i$,
			\item 
			strictly polystable if and only if it intersects all three $\boldsymbol{A}_2$ singularities of $S$, i.e. $H=H_*$,
			\item			unstable if and only if it is a triple line or a cuspidal plane cubic curve whose cuspidal singularity is supported at one of the $\boldsymbol{A}_2$ singularities of $S$, i.e. $H\neq H_*$ but $H\in \mathcal H_i^j$ for some $i,j\in \{1,2,3\}$.
		\end{enumerate}
	\end{lemma}
	\begin{proof}
		We follow the approach in \cite[Sec 9.4]{dolgachev2003lectures}.  The characters of $(\mathbb C^*)^2$ are $\chi((\mathbb C^*)^2)\cong \mathbb Z^2$. The action of $(\mathbb C^*)^2 $ on \eqref{eq:cubic-sections-identification} given in \eqref{eq:torus-action} decomposes $H^0(S, -K_S)$ into four one-dimensional spaces:
		\begin{align}
		H^0\left( S, -K_S \right)
		=
		V_{(1,0)}\oplus V_{(0,0)} \oplus V_{(0,1)} \oplus V_{(-1,-1)},\label{eq:vectorspace-decomposition}
		\end{align}
		where $V_{(a,b)}$ is the eigenspace of vectors invariant by $(t_0^a,t_0^b)$, where $(a,b)\in \chi((\mathbb C^*)^2)\cong \mathbb Z^2$. Indeed
		\begin{align*}
		V_{(1,0)}=\mathbb C \cdot x_0 && V_{(0,0)} = \mathbb C \cdot x_1 && V_{(0,1)}  = \mathbb C \cdot x_2&& V_{(-1,-1)}=\mathbb C \cdot x_3.
		\end{align*}
		Hence, each element of the canonical basis of $H^0(S, -K_S)$ generates an eigenspace for the action of $(\mathbb C^*)^2$.
		
		Given $v\in H^0(S, -K_S)$ its set of weights is
		$$\mathrm{w}(v)=\{\chi\in \mathrm{Hom}((\mathbb C^*)^2, \mathbb C^*)\cong \mathbb Z^2\ : \ v|_{V_\mathcal X}\neq \vec{0} \}.$$
		Given the decomposition \eqref{eq:vectorspace-decomposition} in eigenspaces with for the basis $\{x_0,x_1,x_2,x_3\}$, we have that for any $H=\{\sum_{i=0}^4c_ix_i=0\}|_S\in H^0(S, -K_S)$, $x_i\in \mathrm{w}(H)$ if and only if $c_i\neq 0$. Let $\overline{\mathrm{w}(H)}$ be the convex hull of $\mathrm{w}(H)$. The centroid criterion \cite[Theorem 9.2 9.9]{dolgachev2003lectures} states that $H\in H^0(S, -K_S)$ is GIT stable (GIT semistable, respectively) if and only if $0\in \mathrm{interior}(\overline{\mathrm{w}(H)})$ (if $0\in \overline{\mathrm{w}(H)}$, respectively). Figure \ref{fig:characters} shows the position of the characters in the lattice $\chi((\mathbb C^*)^2)$. 
\begin{figure}[t]
			\begin{tikzpicture}
	\draw[latex-latex, thin, draw=gray] (-2,0)--(2,0) node [right] {}; % l'axe des abscisses
	\draw[latex-latex, thin, draw=gray] (0,-2)--(0,2) node [above] {}; % l'axe des ordonnées

	\foreach \Point/\PointLabel in {(0,1)/{x_2\equiv(0,1)}, (-1,-1)/{x_3\equiv (-1,-1)}}
	\draw[fill=black] \Point circle (0.05) node[above right] {$\PointLabel$};
	\draw[fill=black] (1,0) circle(0.05) node[above right] {${x_0\equiv (1,0)}$};
	\draw[fill=black] (0,0) circle(0.05) node[above left] {${x_1\equiv (0,0)}$};
	\draw [dotted, gray] (-2,-2) grid (2,2);
	
	\end{tikzpicture}
	\caption{Characters of the action $(\mathbb C^*)^2$ on $H^0(S, -K_S)$}
	\label{fig:characters}
\end{figure}
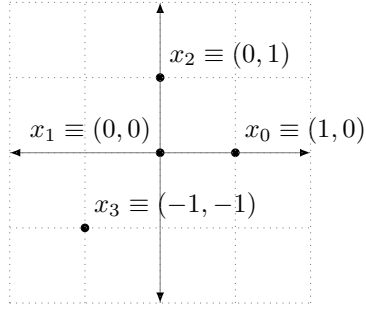		

		Let $H=\{\sum_{i=0}^4c_ix_i=0\}|_S\in \mathbb P(H^0(S, -K_S))$. A direct application of the centroid criterion implies that
		\begin{enumerate}[(i)]
			\item $H$ is stable if and only if $(c_0,c_2,c_3) \in (\mathbb C^*)^3$ and  $c_1 \in \mathbb C$ (i.e. if and only if $\{x_0, x_2, x_3\}= \mathrm{w}(H)$),
			\item 
			$H$ is strictly semistable if and only if one of $x_0,x_2,x_3$ is not in $\mathrm{w}(H)$ and  $x_1\in \mathrm{w}(H)$, i.e. if and only if one the following cases holds:
			\begin{align*}
			c_0=0,  c_1 \in \mathbb C^*, (c_2,c_3) \in \mathbb C^2;
			&&
			c_2=0,  c_1 \in \mathbb C^*, (c_0,c_3) \in \mathbb C^2; &&
			c_3=0,  c_1 \in \mathbb C^*, (c_0,c_2) \in \mathbb C^2,
			\end{align*}
			\item
			$H$ is unstable if and only if one of $x_0,x_2,x_3$ and $x_1\not\in \mathrm{w}(H)$, i.e. if and only if one the following cases holds:
			\begin{align*}
			c_0=c_1=0; & & c_1=c_3=0; & & c_1=c_2=0.
			\end{align*}
		\end{enumerate}
		In particular  $H_*=\{x_1=0\}|_S$ is strictly semistable and $(\mathbb C^*)^2$ acts trivially on it. Hence, $H_*$ is strictly polystable. Moreover, any strictly semistable pair in (ii) can be degenerated to $H_*$ by taking the limit of a one-parameter subgroup of $(\mathbb C^*)^2$. Hence $H_*$ is the only closed strictly semistable orbit, proving (ii) in the statement.
		
		$H$ is a hyperplane section that contains the point $p_i$ if and only if $c_i=0$, where $i=0,2,3$. Hence (i) follows from Lemma \ref{lemma:classification-3A2-anticanonical}. Finally, if $c_1=c_i=0$ for some $i\in \{0,2,3\}$, then $H\in \mathcal H_i$ and by inspection on the proof of Lemma \ref{lemma:classification-3A2-anticanonical}, either $H$ is a triple line, or a cuspidal curve whose cusp lies on an $\boldsymbol{A}_2$ singularity of $S$, proving (iii).
		
\end{proof}
	The above classification implies the following:
	\begin{cor}
		\label{corollary:stability-3A2}
		Consider the GIT quotient with respect to the linearisation $a=b=0$  in \eqref{eq:linear}
		$$
		(\mathbb P(H^0(S, -K_S)))^{ss}\git_{\vec 0}\mathrm{Aut}(S)
		$$
		Let $D\in |-K_S|\cong \mathbb P(H^0(S, -K_S))$, then
		\begin{enumerate}
			\item $D$ is stable if and only if $D \cap Sing(S) =\emptyset$,
			\item $D$ is polystable but not stable if and only if $D=H_*$, the union of the unique three lines in $S$, which contain the three $\boldsymbol{A}_2$ singularities,
			\item $D$ is unstable if and only if $D$ is either a cuspidal plane cubic curve with $Sing(D) \cap Sing(S) \neq \emptyset$ or $D$ is a triple line. 
		\end{enumerate}
	\end{cor}
	\begin{proof}Since 
		$\mathrm{Aut}(S)=(\mathbb C^*)^2 \rtimes \Sigma_3$, we have the exact sequence of groups
		$$
		1 \to (\mathbb C^*)^2 \to  \mathrm{Aut}(S) \to \Sigma_3 \to 1.
		$$
		Let $((t_0, t_2), \sigma)\in \mathrm{Aut}(S)$, where 
		$(t_0,t_2)\in(\mathbb C^*)^2$ and $\sigma \in\Sigma_3$. In particular,
		if $\eta(t)$ is a one-parameter subgroup of $\mathrm{Aut}(S)$, then $\eta(t)=(\lambda (t),\sigma)$, where $\lambda(t)$ is a one-parameter subgroup of $(\mathbb C^*)^2$ and for any $x \in \mathbb P(H^0(S, -K_S)) $
		$$
		\lim_{t \to 0} \eta (t) \cdot x
		=
		\lim_{t \to 0}( \lambda (t) ,\sigma) \cdot x
		=
		\lim_{t \to 0} \lambda (t) \cdot (\sigma \cdot x)
		$$
		Therefore, the hyperplane section $x$ is stable, semi-stable, or unstable with respect to $\eta$ if and only if $\sigma \cdot x$ is stable, semi-stable, unstable with respect to $\lambda$. In particular $x$ is (semi/poly)stable with respect to $\mathrm{Aut}(S)$ if and only if it is (semi/poly)stable with respect to $(\mathbb C^*)^2$. Hence, the result follows from Lemma \ref{lemma:StaC*2}.
	\end{proof}
		\begin{rmk}
	The proof of Lemma \ref{lemma:StaC*2} implies that the projectivisation $\mathbb P\left( (\mathbb C^*)^4 \right)\subset \mathbb P(H^0(S, -K_S))$ of the torus $(\mathbb C^*)^4$ parametrising general elements $H=\{\sum_{i=0}^{4}c_ix_i=0\}$ with all $c_i\neq 0$ is contained in the stable loci of the action linearised by $a=b=0$. Therefore, the GIT quotient is a rational curve i.e isomorphic to $\mathbb P^1$. It is natural to ask how stability changes under a variation of the parameters $(a,b)$. In fact, any such variation will yield at most three possible quotients: The empty space, a point, or $\mathbb P^1$. The latter only takes place when $a=b=0$. Indeed, we can see from Figure \ref{fig:characters} that if either $a\neq 0$ or $b\neq 0$ and $H$ is any hyperplane section, the trivial character $(0,0)$ will either be outside $\overline{w(H)}$ or belong to its boundary. Indeed, any non-trivial value of $(a,b)$ has the effect of translating the characters by the lattice isomorphism $(x,y) \to (x-a,y-b)$. Hence, we can apply the centroid criterion to conclude that any non-trivial choice of $(a,b)$ would render the general hyperplane section either strictly semistable or unstable. As a consequence, the GIT quotient will consist of either a point or it will be empty, respectively. 
	\end{rmk}

		Observe that GIT stability of $D\in |-K_S|$ in Corollary \ref{corollary:stability-3A2} coincides with K-stability of $(S, (1-\beta)D)$ in Proposition \ref{prop:K-stability-cubics-chamber1}. In particular, this verifies Conjecture \ref{conjecture:main} for the cubic surface $S$ with three $\boldsymbol{A}_2$ singularities and $l=1$. This lengthy method of verification can be applied to other cubic surfaces and hyperplane sections. One may expect that a more direct approach may be possible by repeating the arguments of the moduli continuity method. However, due to the technical requirement appearing in Theorem \ref{theorem:K-stability-implies-GIT} (ii), which requires that there is a K-polystable pair $(S, (1-\beta)D)$ such that $D$ is GIT polystable, we implicitly need to use Corollary \ref{corollary:stability-3A2}. In addition, when showing that the image of the continuous map between the Gromov-Hausdorff compactification and the GIT quotient is open and dense, we also require explicit knowledge of the GIT quotient, as we do not work on log smooth pairs any longer. On the other hand, the moduli continuity method provides a stronger result, which allows us to verify Conjecture \ref{conjecture:main} in this special example.
		\begin{prop}
			\label{prop:homeomorphism-3A2-moduli}
			Let $S$ be the unique cubic surface $S$ with three $\boldsymbol{A}_2$ singularities. Let $1>\beta>\beta_0\coloneqq \frac{\sqrt{3}}{2}$. There is a natural homeomorphism between the Gromov-Hausdorff compactification $\overline{M}_{\beta}^{GH}$ of pairs $(S, (1-\beta)D_i)$ where $D_i\in |-K_S|$ and the GIT quotient $\overline M^{GIT}\coloneqq \mathbb P(H^0(S, -K_S))^{ss}\git_{\vec 0} \mathrm{Aut}(S)$ for the natural representation of $\mathrm{Aut}(S)$ in $H^0(S, -K_S)$ linearised by $a=b=0$ as in \eqref{eq:linear}.  
		\end{prop}
		\begin{proof}
			Let $(W, (1-\beta)\Delta)$ be the Gromov-Hausdorff limit of K\"ahler-Einstein pairs $(S, (1-\beta)D_i)$ with $D_i\in |-K_S|$. Let $p\in W$ be a singularity. Since $W$ is a surface, $p$ is klt and hence a quotient singularity, locally analytically isomorphic to $\mathbb C^2/\Gamma$, where $\Gamma$ is a finite group. By flatness, we have that $(-K_W-(1-\beta)\Delta)^2=3\beta^2$. From Theorem \ref{theorem:liu-estimate}, \eqref{eq:volume-inequality-simplification} and Liu's estimate $\widehat{\mathrm{vol}}_{\mathbb C^2/\Gamma, 0}=\frac{4}{\Gamma}$, it follows that  $|\Gamma|\leqslant 3$, so $p$ is an $\boldsymbol{A}_2$ or $\boldsymbol{A}_1$ singularity. Hence, by the classification of cubic surfaces with canonical singularities, $W$ is a cubic surface with at worst $\boldsymbol{A}_2$ singularities, but since it is a degeneration of $S$, we have that $S=W$ and $D\in |-K_S|$. By Lemma \ref{lemma:AutS} we have that $\mathrm{Aut}(S)$ is reductive. By Proposition \ref{prop:K-stability-cubics-chamber1} and Corollary \ref{corollary:stability-3A2} we have that there is some K-polystable pair $(S, (1-\beta)D)$ such that the point $D\in \mathbb P(H^0(S, -K_S))$ is GIT polystable. Hence, by Theorem \ref{theorem:K-stability-implies-GIT} (ii), we can define a map $\phi\colon\overline{M}^{GH}_\beta\rightarrow \overline M^{GIT}$, which sends $(W, (1-\beta)\Delta)$ to $[\Delta]$. As it is customary, injectivity of $\phi$ follows by uniqueness of K\"ahler-Einstein metrics, continuity is a consequence of \cite{Chen-Donaldson-Sun-Kstability-all} and the Luna slice theorem and the image of $\phi$ is compact since $\phi\colon\overline{M}^{GH}_\beta$ is compact. We have that $\phi$ is a homeomorphism if it is surjective, since $\phi$ is a continuous map between a
			compact space and a Hausdorff space. Surjectivity of $\phi$ follows if the image of $\phi$ is open and dense. The latter is a consequence of comparing the descriptions of $\overline{M}^{GH}_\beta$ and $M^{GIT}$ in proposition \ref{prop:K-stability-cubics-chamber1} and Corollary \ref{corollary:stability-3A2}.
		\end{proof}		
		
		\subsection{Behaviour of conical K\"ahler-Einstein metrics near the singularities}
		
		It is natural to ask if one can be more precise on the behaviour of the K\"ahler-Einstein metrics on the singular log pairs which appear in our constructions. The picture is still not complete, mainly due to the missing proofs of the results of \cite{Donaldson-Sun-GH-limits-II} for tangent cones in the case of log pairs. However, we expect the following  generalisations to the log setting of many results known to hold in the absolute case. 
		\begin{conjecture} Let $(W,(1-\beta)\Delta, g)$ be a Gromov-Hausdorff limit of log Fano pairs admitting a K\"ahler-Einstein metric. The following properties hold:
			\begin{enumerate}
				\item For any point $p \in (W,(1-\beta)\Delta, g)$ the metric tangent cone at $p$ is unique (analogous to \cite{Donaldson-Sun-GH-limits-II}).
				\item The metric density  $\Theta_p$ is equal to the normalised volume of the singularity $\widehat{\mathrm{vol}}_{W,p}$ (analogous to \cite{Hein-Sun, Li-Xu-Kollar-Higher-rational-rank}).
				\item The metric tangent cone can be understood algebraic via the two steps construction  (analogous to \cite{Donaldson-Sun, Li-Xu-Kollar-Higher-rational-rank, Li-Wang-Xu-Metric-tangent-cones}).
			\end{enumerate}
		\end{conjecture}
		
		We point out that in some situations the above conjecture is already known to hold by different techniques. In \cite{deBorbon-Spotti-line-arrangements} it is proved that certain log Calabi-Yau metric on surfaces have unique tangent cones given by the expected models, and in such cases it has been checked that metric densities and normalised volumes agree. By applying an implicit function theorem, there are situations in which the log pair is log Fano too (but only close to the Calabi-Yau threshold). 
		
		For a concrete example, note that our theorems give instances of conical K\"ahler-Einstein metrics on pairs $(C,(1-\beta)H)$, where $C$ is a smooth cubic surface and $H$ is an anticanonical divisor with a cuspidal singularity. Since the value of $\beta$ we consider is greater than $5/6$, according to the above conjecture the behaviour of the metric near the singularity of the cusp is given by a tangent cone splitting as $\mathbb{C} \times \mathbb{C}_{\bar\beta}$ with $1-\bar\beta=2(1-\beta)$.
		
	\bibliographystyle{alphatuned}
	\bibliography{bibliography}

\end{document}